\documentclass[leqno,10.5pt]{article} 
\usepackage{amsmath, amssymb}  
\usepackage{amsthm} 
\usepackage{amscd}   
\usepackage{bm}
\usepackage{amsmath, amssymb}    
\usepackage{amsthm} 
\usepackage{ascmac}
\usepackage{amscd}     
\usepackage{color}
\usepackage{graphicx}

\theoremstyle{plain} 
\newtheorem{theorem}{\indent\sc Theorem}[section]
\newtheorem{lemma}[theorem]{\indent\sc Lemma}
\newtheorem{corollary}[theorem]{\indent\sc Corollary}
\newtheorem{proposition}[theorem]{\indent\sc Proposition}

\theoremstyle{definition} 
\newtheorem{definition}[theorem]{\indent\sc Definition}

\newtheorem{remark}[theorem]{\indent\sc Remark}
\newtheorem{example}[theorem]{\indent\sc Example}

\renewcommand{\L}{\mathrm{Li}}

\newcommand{\Q}{\mathbb{Q}}

\newcommand{\Z}{\mathbb{Z}}

\newcommand{\N}{\mathbb{N}}

\def\2{I\hspace{-.1em}I}

\def\deg{{\rm{deg}}}
\def\ord{{\rm{ord}}}
\def\dim{{\rm{dim}}}
\def\ker{{\rm{ker}}}

\def\det{{\rm{det}}}
\def\ord{{\rm{ord}}}
\def\Span{{\rm{Span}}}
\def\Mat{{\rm{Mat}}}

\title{Linear independence of periods related to polylogarithms}
\author{\textsc{Makoto Kawashima}}
\date{ }   
\begin{document}

\maketitle
 


\begin{abstract}
This paper provides the first criteria for the linear independence of multiple polylogarithm values over algebraic number fields. 
In particular, we derive novel results regarding the linear independence of products of polylogarithms at distinct points over an algebraic number field. 
Our approach is based on the explicit construction of Pad\'{e}-type approximants tailored for multiple polylogarithms.
\end{abstract}

\section{Introduction}
Multiple polylogarithms are fundamental functions that arise naturally as periods of families of mixed Hodge structures (cf. \cite{Zhao2, HS}). 
Their special values encompass multiple zeta values, whose arithmetic properties, such as irrationality and linear independence, have been extensively studied yet remain largely shrouded in mystery (see \cite{BF}). 
In this paper, we establish a new criterion for the linear independence of multiple polylogarithms evaluated at algebraic points over number fields in both complex and $p$-adic settings. 
As a corollary, we derive a novel linear independence criterion for products of polylogarithms at distinct algebraic points.

\medskip

A primary tool for investigating the arithmetic of these values is the theory of Pad\'{e} approximation (see \cite{Pade1, Pade2}). 
The classical origins of this approach date back to the work of C.~Hermite \cite{Hermite e} on the transcendence of $e$. 
In $1782$, Legendre introduced the orthogonal polynomials that now bear his name, and in $1816$, Rodrigues discovered a remarkably elegant expression for them, subsequently termed the \emph{Rodrigues formula} by Hermite (cf. \cite{A} for the history of the formula).
Legendre polynomials naturally yield Pad\'{e} approximants for the logarithmic function.

The application of Legendre-type polynomials has led to profound arithmetic consequences. Notably, this approach yielded the irrationality of logarithms (K.~Alladi and M.~L.~Robinson \cite{A-R}) and the irrationality of Riemann zeta values $\zeta(2)$ and $\zeta(3)$ (F.~Beukers \cite{B1, B2, B3}). 
These methods were further enhanced and refined by many authors, including G.~V.~Chudnovsky \cite{ch2}, D.~V.~Chudnovsky and G.~V.~Chudnovsky \cite{Chubrothers84, Chubrothers84 II}, G.~Rhin and P.~Toffin \cite{R-T}, M.~Hata \cite{H1, H2, H3, H4}, V.~N.~Sorokin \cite{S, S1, S2}, R.~Marcovecchio \cite{M}, and C.~Viola and W.~Zudilin \cite{V-Z}, primarily through the discovery of more sophisticated Legendre-type polynomials and their associated integral representations. 
In a related framework, S.~David, N.~Hirata-Kohno, and the author \cite{DHK2, DHK3, DHK4} established linear independence criteria for polylogarithms by deriving Rodrigues-type formulas for their Pad\'{e}-type approximants. 

\medskip

The objective of the present paper is to provide a unified framework for the construction of these Pad\'{e} approximants. 
To this end, we introduce the concept of the \emph{Rodrigues ideal}, which allows us to understand the existence of Rodrigues-type formulas from a purely algebraic perspective. 
By leveraging this algebraic structure, we clarify the essential link between the non-vanishing of certain determinants and the linear independence of the involved functions over the function field.

\medskip

The present paper is divided into two parts. In the first part, we establish a formal general theory of Pad\'{e}-type approximants for holonomic Laurent series. Following earlier ideas from \cite{KP2} (see also \cite{DHK2, KP, Kaw}), we associate with a given Laurent series $f$ a linear map called the \emph{formal $f$-integration transform}, denoted by $\varphi_f$ (see equation~\eqref{varphi f}). We show that the explicit construction of Pad\'{e}-type approximants for $f$ is governed by the structure of the kernel of $\varphi_f$, which can be analyzed through the polynomial coefficient differential operators annihilating $f$ (Corollary~\ref{corollary inc}).

A central pillar of this general theory is the introduction of the \emph{Rodrigues ideal} associated with a polynomial coefficient differential operator $L$. We prove that any nonzero element of this ideal yields Pad\'{e}-type approximants for holonomic Laurent series whose image under $L$ is a polynomial (Proposition~\ref{Pade}). This concept unifies several earlier constructions and provides a transparent algebraic machinery that bypasses the ad-hoc analytic or combinatorial identities often found in the classical literature. By utilizing the algebraic properties of the Rodrigues ideal, we provide a systematic guideline for the explicit construction of approximants, effectively handling the structural complexity of functions such as the multiple polylogarithms treated in Part II.

Furthermore, we address a question of fundamental importance in transcendental number theory: the criteria for the non-vanishing of the determinants of matrices formed by Pad\'{e}-type approximants. We show that, provided there exists a suitable nonzero element in the Rodrigues ideal, the non-vanishing of such a determinant is equivalent to the linear independence of the corresponding family of functions over the function field (Proposition~\ref{equivalence nonzero det}). This equivalence provides a structural explanation of the determinant method underlying many Pad\'{e}-type arguments, establishing a rigorous link between the existence of Rodrigues-type formulas and the functional independence of the series involved.

\medskip

In the second part, we apply the general theory to the case of \emph{multiple polylogarithms}. 
We explicitly construct Pad\'{e}-type approximants for these functions and derive, as an arithmetic application, a \emph{new linear independence criterion for values of multiple polylogarithms} over an algebraic number field (Theorem~\ref{main}). 
As a corollary, we establish the linear independence of products of polylogarithms at distinct points (Corollary~\ref{main cor}).

Furthermore, while Theorem~\ref{main} can be viewed as a generalization of the work of Sorokin \cite{S1, S2}, it should be noted that his papers, particularly \cite{S2}, did not provide explicit statements regarding the linear independence of multiple polylogarithms or products of polylogarithms at distinct points.  
Specifically, Sorokin's approach heavily relied on the analytic properties of the underlying functions, known as the Nikishin system (defined in \cite{Nikishin}), to establish the non-vanishing of certain crucial determinants. 
In contrast, one of the key contributions of the present paper is to eliminate these analytic constraints by introducing a purely formal treatment of Pad\'{e}-type approximations. 
By characterizing the approximation properties via intrinsic algebraic structures, we are able to extend the linear independence assertions unconditionally over arbitrary algebraic number fields.

\medskip

By applying the methodology of this paper to other Legendre-type polynomials, it is also possible to construct and analyze Pad\'{e}-type approximants for other functions related to multiple polylogarithms. 
As an example, in Section~\ref{section: Pade power of log}, we provide new Pad\'{e}-type approximants for products of logarithms. 
As observed by the Chudnovsky brothers (cf.~\cite[Sections 4, 5, and 6]{Chu bro 1999}), these Pad\'{e}-type approximants are expected to satisfy linear recurrence relations of Poincar\'{e} type.
If the asymptotic behavior of the magnitude of these linear recurrences can be determined using the Poincar\'{e}-Perron theorem, 
it will allow for a precise analysis of the linear independence measures of multiple polylogarithm values.

\section{Basic definitions and the statement of the main results}
We begin by establishing the notation and conventions used throughout this article. 
Let $K$ be an algebraic number field of degree $d = [K : \mathbb{Q}]$. 
We denote the set of all places of $K$ by $\mathfrak{M}_K$, which is partitioned into the set of finite places $\mathfrak{M}_K^f$ and infinite places $\mathfrak{M}_K^\infty$. 
For each place $v \in \mathfrak{M}_K$, let $K_v$ denote the completion of $K$ at $v$, and let $d_v = [K_v : \mathbb{Q}_p]$ be the local degree, where $p$ is the place of $\mathbb{Q}$ lying below $v$.
The normalized absolute value $|\cdot|_v$ on $K$ is defined as follows:
\begin{itemize}
    \item If $v \in \mathfrak{M}_K^f$ lies above $p$, then $|p|_v = p^{-\tfrac{d_v}{d}}$.
    \item If $v \in \mathfrak{M}_K^\infty$ corresponds to an embedding $\sigma_v : K \hookrightarrow \mathbb{C}$, then $|x|_v = |\sigma_v(x)|^{\tfrac{d_v}{d}}$.
\end{itemize}
Let $m$ be a nonnegative integer and $\boldsymbol{\beta} = (\beta_0, \dots, \beta_m) \in K^{m+1}$. 
The absolute affine height of $\boldsymbol{\beta}$ is defined by
\[
\mathrm{H}(\boldsymbol{\beta}) = \prod_{v \in \mathfrak{M}_K} \max\{1, |\beta_0|_v, \dots, |\beta_m|_v\},
\]
and the logarithmic absolute height is given by $h(\boldsymbol{\beta}) = \log \mathrm{H}(\boldsymbol{\beta})$. 
For each place $v \in \mathfrak{M}_K$, we define the local contribution to the height as
\[
h_v(\boldsymbol{\beta}) = \log \max\{1, |\beta_0|_v, \dots, |\beta_m|_v\},
\]
so that the global height admits the decomposition
\[
h(\boldsymbol{\beta}) = \sum_{v \in \mathfrak{M}_K} h_v(\boldsymbol{\beta}).
\]
For $v \in \mathfrak{M}_K$, let $\varepsilon_v = 1$ if $v \mid \infty$ and $\varepsilon_v = 0$ if $v \nmid \infty$. 

\medskip

We denote the set of positive integers by $\N$. 
Moreover, for $\boldsymbol{s}=(s_1,\ldots,s_d)\in \N^d$ we set $|\boldsymbol{s}|=\sum_{i=1}^ks_i$. 
Let us recall the definition of multiple polylogarithms introduced by A.~B.~Goncharov in \cite{Goncharov}.
\begin{definition}
Let $k$ be a positive integer, $\boldsymbol{s}=(s_1, \dots, s_k)\in\N^k$, and let $z_1, \dots, z_k$ be variables. 
For $\boldsymbol{s}$, the multiple polylogarithm is defined by
\[
\L_{\boldsymbol{s}}(z_1, \dots, z_k) = \sum_{0<n_1<n_2< \dots < n_k } \frac{z_1^{n_1} z_2^{n_2} \dots z_k^{n_k}}{n_1^{s_1} n_2^{s_2} \dots n_k^{s_k}}\in \Q[[z_1,\ldots,z_k]].
\]
Conventionally one refers $k$ as the {\it depth} and $|\boldsymbol{s}|$ as the {\it weight}. 
When the depth $k=1$, the function is nothing but the classical polylogarithms.
Note that for any place $v$ of $K$, the multiple polylogarithm $\L_{\boldsymbol{s}}(z_1, \dots, z_k)$ converges to an element of $K_{v}$ in the domain $|z_j|_v < 1$ for all $j = 1, \dots, k$.
\end{definition}

\medskip
 
We are now in a position to state our main results. 
Let $m,r$ be positive integers. 
Put $M=(m+1)^r-1$. 
We fix a vector $\boldsymbol{\alpha}=(\alpha_1,\ldots,\alpha_m)\in (K^{\times})^m$ whose coordinates are pairwise distinct.
For $\beta \in K$ such that $|\beta|_v > H_{v}(\boldsymbol{\alpha})$, we define the following real number:
\begin{align}
V_{v}(\boldsymbol{\alpha},\beta) &= (M+1)h_v(\beta) - h_v(\boldsymbol{\alpha})-M \left( h(\beta) + \frac{1}{m}\sum_{i=1}^m h(\alpha_i) + h(\boldsymbol{\alpha}) \right) \label{V} \\
&\quad - \left( M\log 2 + \frac{r(r+1)}{2}\log(m+1) + r + rM \right). \nonumber 
\end{align} 
The following theorem provides a novel linear independence criterion of the values of multiple polylogarithm over an algebraic number field. 
\begin{theorem} \label{main}
Let $v_0\in \mathfrak{M}_K$ and $\beta\in K$ with $|\beta|_{v_0}>H_{v_0}(\boldsymbol{\alpha})$.
Assume $V_{v_0}(\bm{\alpha},\beta)>0$. {\footnote{The positivity of $V_{v_0}(\boldsymbol{\alpha},\beta)$ signifies that $\beta$ possesses sufficient arithmetic magnitude at $v_0$ relative to the heights of the parameters. This condition is crucial for ensuring the convergence of the relevant Laurent series at $z=\beta$ and for maintaining the arithmetic stability necessary for Siegel's method.
}} 
Then, the following $M$ values of multiple polylogarithm in $K_{v_0}$$:$
\[
\L_{\boldsymbol{s}}(\alpha_{i_1}/\alpha_{i_2},...,\alpha_{i_{k-1}}/\alpha_{i_k},\alpha_{i_k}/\beta)\footnote{These values can also be expressed in terms of hyperlogarithms (see \cite{Zhao2}).} \ \ \text{for} \ \  \boldsymbol{s}=(s_1,\ldots,s_k)\in \cup_{k=1}^r\N^k \ \text{with} \ |\boldsymbol{s}|\le r  \ \text{and} \  1\le i_j \le m,
\]
together with $1$ are linearly independent over $K$.
\end{theorem}
The following corollary gives a new linear independence criterion of the values of product of polylogarithms at distinct points.
\begin{corollary} \label{main cor}
We keep the notation in Theorem~$\ref{main}$. 
Assume $V_{v_0}(\boldsymbol{\alpha},\beta)>0$.
Then the subset of $K_{v_0}$ of products of the polylogarithms$:$
\[
\{1\}\cup \{\L_{s_1}(\alpha_{i_1}/\beta)\ldots\L_{s_k}(\alpha_{i_k}/\beta) \mid   \boldsymbol{s}=(s_1,\ldots,s_k)\in \cup_{k=1}^r\N^k \ \text{with} \ |\boldsymbol{s}|\le r  \ \text{and} \  1\le i_j \le m\}\footnote{Note that although different sequences of indices may represent the same product due to commutativity, this is considered as a set of distinct values.},
\]
is linearly independent over $K$.
\end{corollary}

\medskip

\noindent\textbf{Outline of the article.} 
In Section~\ref{general Pade}, we begin by introducing the notion of Pad\'{e}-type approximants for Laurent series.  
In Section~\ref{formal f-integration}, we introduce the \textit{formal $f$-integration transform} associated with a Laurent series $f$.
 This transform plays a central role throughout the paper; we describe its fundamental properties, particularly in the case where $f$ is holonomic.  
In Section~\ref{rodrigues formula}, we define the Rodrigues ideal associated with a differential operator $L$ with polynomial coefficients. 
We demonstrate that any non-zero element of this Rodrigues ideal provides Pad\'{e}-type approximants for Laurent series that are mapped to polynomials under the action of $L$.
In Section~\ref{section: linear independence pade}, we prove that, given a suitable non-zero element in the Rodrigues ideal, 
the non-vanishing of the associated determinant is equivalent to the linear independence of the corresponding family of functions together with $1$ over the function field.   

Beginning in Section~\ref{section: Pade multiple polylog}, we apply the results from Sections~\ref{rodrigues formula} and~\ref{section: linear independence pade} to a specific class of multiple polylogarithms.
Section~\ref{section: Pade multiple polylog} is devoted to the explicit construction of Pad\'{e}-type approximants for these functions.
In Section~\ref{section: estimates}, we establish several key estimates, including the growth of the Pad\'{e}-type approximants and their corresponding remainder terms for both Archimedean and non-Archimedean valuations.
Section~\ref{section: proof of main theorem} contains the proof of Theorem~\ref{main} and Corollary~\ref{main cor}.  
Finally, in Section~\ref{section: Pade power of log}, we provide a novel example of Pad\'{e} approximants for powers of logarithms. 

\section{Pad\'{e}-type approximants of Laurent series} \label{general Pade}
Throughout this section, we fix a field $K$ of characteristic $0$.
We denote the formal power series ring of variable $1/z$ with coefficients $K$ by $K[[1/z]]$ and the field of fractions by $K((1/z))$. We say an element of $K((1/z))$ is a formal Laurent series.
We define the order function at $z=\infty$ by
$${\rm{ord}}_{\infty}:K((1/z)) \longrightarrow \Z\cup \{\infty\}; \ \sum_{k} \dfrac{a_k}{z^k} \mapsto \min\{k\in\Z\cup \{\infty\} \mid a_k\neq 0\}.$$
Remark that, for $f\in K((1/z))$, ${\rm{ord}}_{\infty} \, f=\infty$ if and only if $f=0$.
\begin{lemma} \label{pade}
Let $m$ be a nonnegative integer, $f_1(z),\ldots,f_m(z)\in 1/z\cdot K[[1/z]]$ and $\boldsymbol{n}=(n_1,\ldots,n_m)\in \N^{m}$.
Put $N=\sum_{j=1}^mn_j$.
For a nonnegative integer $M$ with $M\ge N$, there exist polynomials $(P,Q_{1},\ldots,Q_m)\in K[z]^{m+1}\setminus\{\boldsymbol{0}\}$ satisfying the following conditions$:$
    \begin{enumerate}
      \item[{\rm{\rm(i)}}]  \label{item: lem pade: item 1} $\deg \, P\le M$,
      \item[{\rm(ii)}]  \label{item: lem pade: item 2} ${\rm{ord}}_{\infty} \left(P(z)f_j(z)-Q_j(z)\right)\ge n_j+1$ for $j=1,\dots,m$.
    \end{enumerate}
\end{lemma}
\begin{definition}
We say that a vector of polynomials $(P,Q_{1},\ldots,Q_m) \in K[z]^{m+1}$ satisfying the properties ${\rm(i)}$ and ${\rm(ii)}$ is weight $\boldsymbol{n}$ and degree $M$ Pad\'{e}-type approximants\footnote{In case of $\boldsymbol{n}=(n,\ldots,n)$, we say weight $n$ and degree $M$ Pad\'{e}-type approximant for short.} of 
$(f_1,\ldots,f_m)$.
For such approximants $(P,Q_{1},\ldots,Q_m)$ of $(f_1,\ldots,f_m)$, we call the formal Laurent series $(P(z)f_j(z)-Q_{j}(z))_{1\le j \le m}$, \textit{id est} remainders, as weight $\boldsymbol{n}$ degree $M$ Pad\'{e}-type approximations of $(f_1,\ldots,f_m)$.
\end{definition}

\section{Formal $f$-integration transform} \label{formal f-integration}
Let $f(z)=\sum_{k=0}^{\infty} f_k/z^{k+1}\in 1/z\cdot K[[1/z]]$.
We define a {{$K$-linear map}} $\varphi_f\in {\rm{Hom}}_K(K[t],K)$ by
\begin{align} \label{varphi f}
\varphi_f:K[t]\longrightarrow K; \ \ \ t^k\mapsto f_k \ \ \ (k\ge0).
\end{align}
Note that the $K$-linear map
\begin{align} \label{Phi}
    \Phi: 1/z\cdot K[[1/z]]\longrightarrow {\rm{Hom}}_K(K[t],K)
\end{align}
defined by $f\mapsto \varphi_f$ is an isomorphism.

\medskip

The above linear map extends naturally 
in a $K((1/z))$-{{linear map}} $\varphi_f: K((1/z))[t]\rightarrow K((1/z))$. With this notation, the formal Laurent series $f(z)$ satisfies the following crucial identities {{(cf. \cite[$(6.2)$ p. 60 and $(5.7)$ p.52]{N-S})}}:
\begin{align*}
&f(z)=\varphi_f \left(\dfrac{1}{z-t}\right)\enspace,\ \ \ P(z)f(z)-\varphi_f\left(\dfrac{P(z)-P(t)}{z-t}\right)\in 1/z\cdot K[[1/z]] \ \ \text{for any} \ \ P(z)\in K[z]\enspace.
\end{align*}

The following lemma provides an equivalent condition for a polynomial to be a Pad\'{e}-type approximant for a family of Laurent series, utilizing the formal integration transform.
\begin{lemma} {\rm{\cite[Lemma~2.1]{Kaw}}}\label{equivalence Pade}
Let $m$ be a nonnegative integer, $f_1(z),\ldots,f_m(z)\in 1/z\cdot K[[1/z]]$ and $\boldsymbol{n}=(n_1,\ldots,n_m)\in \N^m$. Let $M$ be a positive integer and $P(z)\in K[z]$ a nonzero polynomial with 
$M\ge \sum_{j=1}^mn_j$ and ${\rm{deg}}\,P\le M$. Put $Q_j(z)=\varphi_{f_j}\left(\tfrac{P(z)-P(t)}{z-t}\right)\in K[z]$ for $1\le j \le m$.
Then the following statements are equivalent.
\begin{enumerate}
      \item[{\rm{\rm(i)}}] The vector of polynomials $(P,Q_1,\ldots,Q_m)$ is a weight $\boldsymbol{n}$ Pad\'{e}-type approximants of $(f_1,\ldots,f_m)$.
      \item[{\rm{\rm(i)}}] We have $t^kP(t)\in {\rm{ker}}\,\varphi_{f_j}$ for $1\le j \le m$, $0\le k \le n_j-1$.
\end{enumerate}
\end{lemma}
Lemma~\ref{equivalence Pade} suggests that the study of $\ker\, \varphi_f$ is essential for the explicit construction of Pad\'{e}-type approximants to Laurent series. We now investigate $\ker \,\varphi_f$ for a holonomic Laurent series $f \in 1/z\cdot K[[1/z]]$. Throughout this section, we denote the differential operators $d/dz$ and $d/dt$ by $\partial_z$ and $\partial_t$, respectively. The action of a differential operator $L$ on a function $f$ is denoted by $L \cdot f$, with $\partial_z \cdot f$ abbreviated as $f'$.
We view the elements of the Weyl algebra $K[t,\partial_t]$ as $K$-endomorphisms of the polynomial ring $K[t]$ via the natural embedding $K[t,\partial_t]\hookrightarrow {\rm{End}}_K(K[t])$. 

\medskip

We begin by introducing the \textit{formal adjoint map}
\begin{equation} \label{iota diff}
\iota : K(z)[\partial_z] \longrightarrow K(t)[\partial_t], \quad \sum_j P_j(z)\partial_z^j \longmapsto \sum_j (-1)^j \partial_t^j P_j(t).
\end{equation}
Remark, for $L\in K(z)[\partial_z]$, $\iota(L)$ is called the {\it formal adjoint} of $L$ and related to the dual of differential module $K(z)[\partial_z] / K(z)[\partial_z]L$ (cf.~\cite[III, Exercise 3]{An1}).
For $L\in K(z)[\partial_z]$, we denote $\iota(L)$ by $L^{*}$. 
Note that the map $\iota$ is an anti-isomorphism of Weyl algebras, satisfying $(L_1 L_2)^* = L_2^* L_1^*$ for any $L_1, L_2 \in K(z)[\partial_z]$.
We also introduce the projection morphism $\pi$ by
\begin{align*} 
\pi :K((1/z))\longrightarrow K((1/z))/K[z]\cong 1/z\cdot K[[1/z]]; \ \ \ f(z)=P(z)+\tilde{f}(z)\mapsto \tilde{f}(z)\enspace, 
\end{align*}
where $P(z)\in K[z]$ and $\tilde{f}(z)\in 1/z\cdot K[[1/z]]$.

We remark that for any differential operator $L\in K[z,\partial_z]$ and Laurent series $f\in K((1/z))$, 
\begin{align} \label{pi circ L}
\pi(L\cdot f)=\pi(L\cdot \pi(f)).
\end{align}

We recall a fundamental result that is central to our investigation of Pad\'{e}-type approximants.
\begin{proposition} {\rm{\cite[Proposition~2.5]{Kaw}}} \label{key prop}
Let $L\in K[z,\partial_z]$ and $f(z)\in 1/z\cdot K[[1/z]]$. Then we have 
\[
\varphi_{\pi(L\cdot f)}=\varphi_f\circ L^{*}.
\]
\end{proposition}
Proposition \ref{key prop} immediately yields the following characterization, which establishes a crucial equivalence relation.
\begin{corollary} {\rm{\cite[Corollary~$2.6$]{Kaw}}} \label{corollary inc} 
Let $f(z)\in 1/z\cdot K[[1/z]]$ and $L\in K[z,\partial_z]$.
The following are equivalent.
\begin{enumerate}
      \item[{\rm{\rm(i)}}] $L\cdot f\in K[z]$. 
      \item[{\rm(ii)}] $L^{*}\cdot K[t]\subseteq {\rm{ker}}\,\varphi_f$.
\end{enumerate}
\end{corollary}
\begin{proof} 
For the sake of completeness, we recall the proof of the statement. 
The conditions ${\rm(i)}$, ${\rm(ii)}$ are equivalent to $\pi(L\cdot f)=0$ and $\varphi_f \circ L^{*}=0$ respectively. Therefore by Proposition \ref{key prop}, we obtain the assertion. 
\end{proof}
\section{Rodrigues formula} \label{rodrigues formula}
We keep the notation in Section~\ref{formal f-integration}. 
We prepare further notation.
Let $K$ be a field of characteristic $0$.
For a $K$-vector space $V$ and the subset $S\subset W$, we denote the $K$-vector subspace of $V$ generated by $S$ by $\Span_K\, S$.
Given an integer $n$, we denote by $K[z]$ the ring of polynomials in $z$ with coefficients in $K$, and by $K[z]_{\leq n} \subset K[z]$ the subgroup of polynomials of degree at most $n$ with the convention $K[z]=\{0\}$ if $n<0$.
\begin{definition}
Let $L=\sum_{j=0}^m (-1)^ja_j(z)\partial_z^j \in K[z,\partial_z]$ with $a_m(z)\neq 0$. 
We assign weights $1$ and $-1$ to $z$ and $\partial_z$ respectively. The {\emph{order}} of $L$ with respect to the weight $(1,-1)$ is defined as 
\[
\ord_{(1,-1)}(L)=\max_{0\le j\le m}\{\deg\,a_j-j\}.
\]
Note that the function $\ord_{(1,-1)}$ on $K[z,\partial_z]$ is additive (see \cite[1.2]{SST}); that is, for any $L_1,L_2\in K[z,\partial_z]$, 
\begin{align} \label{sum of ord}
\ord_{(1,-1)}(L_1L_2)=\ord_{(1,-1)}(L_1)+\ord_{(1,-1)}(L_2).
\end{align}
\end{definition}

Let $n$ be a positive integer.  
Define the $K$-vector space associated with $L$ by
\begin{align*}
V_1(L)&:=\{ f\in 1/z\cdot K[[1/z]] \mid L\cdot f\in K[z] \},\\
V_n(L)&:= {\rm Span}_K\{ \pi(z^kf) \mid  0\le k\le n-1,\ f\in V_1(L)\} \ \ (n\ge2).
\end{align*}
If there are no confusion, we denote $V_n(L)=V_n$.
Note that $\dim_KV_n(L)\le n\cdot \dim_KV_1(L)$.

\medskip

The aim of this section is to show Rodrigues formula for the Pad\'{e}-type approximants of elements of $V_1(L)$ (see Proposition~\ref{Pade}). 
First we study the $K$-vector spaces $V_n(L)$. 

\begin{lemma} \label{lem:linear_indep_equiv}
Let $(f_j)_{1 \le j \le d}$ be a $K$-basis of $V_1(L)$. Then the following are equivalent$:$
\begin{enumerate}
      \item[\rm(i)] For every $n \ge 1$, $\dim_K V_n(L) = n \cdot \dim_K V_1(L)$.
      \item[\rm(ii)] The elements $1, f_1, \dots, f_d$ are linearly independent over $K(z)$.
\end{enumerate}
\end{lemma}

\begin{proof}
Set $d = \dim_K V_1(L)$. If $d=0$, the statement is trivial, so we assume $d \ge 1$.  
For each $n \in \mathbb{N}$, consider the set
\[
S_n := \{ \pi(z^k f_j) \mid 0 \le k \le n-1, \ 1 \le j \le d \},
\]
which spans $V_n(L)$ over $K$. Since $\#S_n\le dn$, the condition ${\rm{dim}}_K V_n(L) = dn$ in ${\rm{(i)}}$ is equivalent to the assertion that $\#S_n=dn$ and $S_n$ forms a $K$-basis of $V_n(L)$ for all $n \ge 1$.

\medskip 
\noindent \textbf{(i) $\implies$ (ii):} 
Suppose $1, f_1, \dots, f_d$ are linearly dependent over $K(z)$. Then there exists a non-zero vector of polynomials $(P_0, P_1, \dots, P_d) \in K[z]^{d+1} \setminus \{\mathbf{0}\}$ such that 
\begin{equation} \label{eq:linear_relation}
P_0(z) + \sum_{j=1}^d P_j(z) f_j = 0.
\end{equation}
Let $n-1 = \max_{1 \le j \le d} \{ \deg P_j \}$ and write $P_j(z) = \sum_{k=0}^{n-1} p_{j,k} z^k$. Note that if all $P_1, \dots, P_d$ were zero, then $P_0$ would also be zero, contradicting our assumption. Applying the projection $\pi$ to equation \eqref{eq:linear_relation}, we obtain
\[
\sum_{j=1}^d \sum_{k=0}^{n-1} p_{j,k} \pi(z^k f_j) = 0.
\]
By the $K$-linear independence of $S_n$ (which follows from (i)), we must have $p_{j,k} = 0$ for all $j, k$. Thus $P_1 = \dots = P_d = 0$, which implies $P_0 = 0$, a contradiction.

\medskip
\noindent \textbf{(ii) $\implies$ (i):} 
Conversely, if (i) fails for some $n$, the set $S_n$ is $K$-linearly dependent. Then there exist $p_{j,k} \in K$, not all zero, such that $\sum_{j,k} p_{j,k} \pi(z^k f_j) = 0$. This implies $P_0 + \sum_{j=1}^d P_j f_j = 0$ for some $P_0 \in K[z]$, where $P_j(z) = \sum p_{j,k} z^k$ are not all zero. This contradicts (ii).
\end{proof}

Let $L=\sum_{j=0}^m (-1)^ja_j(z)\partial_z^j \in K[z,\partial_z]$ with $a_m(z)\neq 0$. 
Now we put $\deg\,a_j=m_j$ and
\begin{align} \label{L aj}
a_j(z)=\sum_{i=0}^{m_j}a_{i,j}z^i.
\end{align}
Consider the linear recurrence relations:
\begin{align} \label{recurrence}
\sum_{j=0}^m\sum_{i=\max\{0,j-k\}}^{m_j}a_{i,j}(k+i-j+1)_jx_{k+i-j}=0 \ \ \ (k=0,1,2,\ldots,).
\end{align}
Denote the $K$-space of solutions of equation~\eqref{recurrence} by 
\[
\mathcal{V}_1(L)=\{(x_k)_{k\ge0}\in K^{\N} \mid (x_k)_k \ \text{satisfies equation}~\eqref{recurrence}.\}.
\]
\begin{lemma} \label{cond dim>0}
The $K$-morphism:
\[
V_1(L)\longrightarrow \mathcal{V}_1(L); \ f(z)=\sum_{k=0}^{\infty}\dfrac{f_k}{z^{k+1}}\mapsto (f_k)_{k\ge0},
\]
is an isomorphism.
\end{lemma}
\begin{proof}
Put $d=\ord_{(1,-1)}(L)$ and let
\[
f(z)=\sum_{k=0}^{\infty}\frac{f_k}{z^{k+1}}\in \frac{1}{z}\,K[[1/z]].
\]
A straightforward computation shows that
\[
L\cdot f(z)
= A(z)
+ \sum_{k=0}^{\infty}
\frac{
\sum_{j=0}^m
\sum_{i=\max\{0,j-k\}}^{m_j}
a_{i,j}(k+i-j+1)_j\, f_{k+i-j}
}{z^{k+1}},
\]
where
\[
A(z)
=
\sum_{j=0}^m
\sum_{i=j+1}^{m_j}
\sum_{k=0}^{i-j-1}
a_{i,j}(k+1)_j\, f_k\,
z^{\,i-j-k-1}
\in K[z]_{\le d-1}.
\]
Consequently, the condition $f(z)\in V_1(L)$ is equivalent to the fact that the sequence
$(f_k)_{k\ge 0}$ satisfies the recurrence relation \eqref{recurrence} for all $k\ge 0$.
This completes the proof.
\end{proof}
\begin{definition} \label{Rodrigues ideal}
For a positive integer $n$, we define the left ideal $I_n(L) \subset K[z, \partial_z]$ associated with $L$ as
\[
I_n(L) := \{ R \in K[z, \partial_z] \mid R \cdot f \in K[z] \text{ for all } f \in V_n(L)\}.
\]
The \textit{$n$th Rodrigues ideal} associated with $L$, denoted by $I_n^*(L)$, is the right ideal of $K[t, \partial_t]$ obtained as the image of $I_n(L)$ under the formal adjoint map $\iota$ defined in \eqref{iota diff}:
\[
I^*_n(L) := \{ \iota(R) = R^* \mid R \in I_n(L) \}.
\]
\end{definition}
\begin{lemma} \label{equivalence R in I_n(L)}
Let $R\in K[z,\partial_z]$. The following are equivalent. 
\begin{enumerate}
      \item[{\rm(i)}] The operator $R$ belongs to $I_n(L)$.
      \item[{\rm(ii)}] For any integer $k$ with $0\le k \le n-1$ and $f\in V_1(L)$, we have $t^kR^*\cdot K[t]\subseteq \ker\,\varphi_f$.
\end{enumerate}
\end{lemma}
\begin{proof}
Let $k$ be an integer with $0\le k \le n-1$ and $f\in V_1(L)$.
Corollary~\ref{corollary inc} implies that the inclusion $t^kR^*\cdot K[t]\subseteq  \ker\,\varphi_f$ is equivalent to $Rz^k\cdot f\in K[z]$, which is in turn equivalent to $R\cdot \pi(z^kf)\in K[z]$. 
These equivalences establish the relationship between ${\rm(i)}$ and ${\rm(ii)}$.
\end{proof}
\begin{example} \label{order 1}
Let us take order $1$ differential operator $L=-a_1(z)\partial_z+a_0(z)\in K[z,\partial_z]$. For a positive integer $n$, we define
\[
R_n=a^n_1(z)\left(-\partial_z+\dfrac{a_0(z)}{a_1(z)}\right)^n.
\]
The classical Rodrigues formula states that, the formal adjoint of $R_n$,
\[
R^{*}_n=\left(\partial_t+\dfrac{a_0(t)}{a_1(t)}\right)^na^n_1(t)
\]
belongs to the $n$th Rodrigues ideal $I^*_n(L)$ (see \cite[Theorem~4.2]{Kaw}).
\end{example}

In the next proposition, we show that for any nonzero $R\in I_n(L)$, the formal adjoint $R^*$ yields Pad\'{e}-type approximants for any $K$-basis of $V_1(L)$.
We define the evaluation map ${\rm{Eval}}_{t=z}$ by
\[
{\rm{Eval}}_{t=z}: K[t]\longrightarrow K[z]; \ P(t)\mapsto P(z).
\]
\begin{proposition}\label{Pade}
Assume that $\dim_K V_1(L)>0$, and set $d:=\dim_K V_1(L)$.
\begin{enumerate}
      \item[{\rm{\rm(i)}}] For every positive integer $n$, the left ideal $I_n(L)$ is nonzero.
      \item[{\rm(ii)}] Let $f_1,\ldots,f_d$ be a $K$-basis of $V_1(L)$, and let $A(t)\in K[t]$.  
Let $R\in I_n(L)\setminus\{0\}$. For a nonzero element $R^{*}$ of the $n$th Rodrigues ideal associated with $L$, we define polynomials
\[
P(z):=\mathrm{Eval}_{t=z}\left( R^*\cdot A(t)\right), 
\qquad 
Q_j(z):=\varphi_{f_j}\!\left(\frac{P(z)-P(t)}{z-t}\right)
\quad (1\le j\le d).
\]
Assume that $P(z)\neq 0$.  
Then the $(d+1)$-tuple of polynomials $(P,Q_1,\ldots,Q_d)$ is a weight $n$ Pad\'{e}-type approximant of $(f_1,\ldots,f_d)$.
      \item[{\rm(iii)}] Let
\[
\mathfrak{R}_j(z):=P(z)f_j(z)-Q_j(z)
\qquad (1\le j\le d)
\]
denote the corresponding Pad\'{e}-type remainders.  
Then we have
\[
\mathfrak{R}_j(z)
 = \sum_{k=n}^{\infty}\frac{\varphi_{f_j}(t^kP(t))}{z^{k+1}},
 \qquad 1\le j\le d.
\]
\end{enumerate}
\end{proposition}
\begin{proof}
${\rm(i)}$ Note that the Weyl algebra $K[z,\partial_z]$ is a left Ore domain; that is, any two nonzero left ideals $I,J$ of $K[z,\partial_z]$ satisfy $I\cap J\neq 0$ (see \cite[$8.4$~Proposition]{Bj}).
For a nonnegative integer $k$, we define the left ideal $J_k(L)\subseteq K[z,\partial_z]$ by 
\[
J_k(L)=\{R\in K[z,\partial_z]\mid R\cdot z^kf\in K[z] \ \text{for all} \ f\in V_1(L)\}.
\]
Notice that the nonzero differential operator $z^{m+k}Lz^{-k}$ belongs to the left ideal $J_k(L)$. Applying equation \eqref{pi circ L}, we obtain
\[
I_n(L)=\bigcap_{k=0}^{n-1} J_k(L).
\]
Consequently, the left ideal $I_n(L)$ is also nonzero, which completes the proof. 

\medskip
\noindent
${\rm(ii)}$
By Lemma~\ref{equivalence Pade}, it suffices to show that
\begin{align} \label{orthogonal fj}
\varphi_{f_j}(t^kP(t))=0
\qquad (0\le k\le n-1).
\end{align}
Applying the identity $P(t)=R^{*}\cdot A(t)$ and invoking condition $({\rm{ii}})$ from Lemma~\ref{equivalence R in I_n(L)}, we obtain:
\begin{align*} 
\varphi_{f_j}(t^kP(t))
 = \varphi_{f_j}\circ t^k R^*(A(t)) 
 = 0,
\end{align*}
which completes the proof.

\medskip
\noindent
${\rm(iii)}$ Using the identity $f_j(z)=\varphi_{f_j}\left(1/(z-t)\right)$ together with the definition of $Q_j$, we have
\[
\mathfrak{R}_j(z) = \varphi_{f_j}\!\left(\frac{P(t)}{z-t}\right).
\]
Expanding the right-hand side via the identity $1/(z-t)=\sum_{k=0}^{\infty}t^k/z^{k+1}$ with equation~\eqref{orthogonal fj} yields 
\[
\mathfrak{R}_j(z)
 = \sum_{k=n}^{\infty}\frac{\varphi_{f_j}(t^kP(t))}{z^{k+1}},
\]
which completes the proof.
\end{proof}

\medskip

In the following, let us consider a condition of the differential operator $L$ so that $\dim_KV_1(L)=\ord_{(1,-1)}(L)$ and some properties of such operator.
\begin{definition} \label{property P}
We use the notation defined in \eqref{L aj}. Put $d=\ord_{(1,-1)}(L)$ and assume $d\ge 1$.  
We say that the differential operator $L$ has \emph{property $(P)$} if the following holds:
\[
\sum_{\substack{0\le j \le m\\ m_j-j=d}}a_{m_j,j}(k+d+1)_{m_j}\neq 0 \ \text{for any} \ k\ge0.
\]
\end{definition}

\begin{lemma} \label{equiv of (P)}
Let $L\in K[z,\partial_z]$. Put $d=\ord_{(1,-1)}(L)$ and assume $d\ge 1$. The following are equivalent.
\begin{enumerate}
      \item[{\rm{\rm(i)}}] $L$ has property $(P)$.
      \item[{\rm(ii)}] For every polynomial $P(t)\in K[t]$, we have $\deg\, L^*\cdot P = {\rm deg}\,P + d$.
\end{enumerate}
\end{lemma}
\begin{proof}
Since we have
\[
L^{*}=\sum_{j=0}^{m}\partial^j_t\sum_{i=0}^{m_j}a_{i,j}t^j\in K[t,\partial_t],
\]
for any nonnegative integer $k$, we have
\begin{align*}
L^{*}(t^k)&=\sum_{j=0}^m\sum_{i=0}^{m_j}a_{i,j}(k+i-j+1)_jt^{k+i-j}\\
&=\sum_{\substack{0\le j \le m\\ m_j-j=d}}a_{m_j,j}(k+d+1)_{m_j}t^{k+d}+(\text{lower degree terms}).
\end{align*}
These equalities imply that ${\rm(i)}$ is equivalent to ${\rm(ii)}$.
\end{proof}

\begin{lemma} \label{property P implies}
Let $L\in K[z,\partial_z]$. Put $d=\ord_{(1,-1)}(L)$. Assume $d\ge 1$ and $L$ satisfies property $(P)$.  
Then the following statements hold.
\begin{enumerate}
      \item[{\rm{\rm(i)}}]  $\dim_K V_1(L) = d$. 
      \item[{\rm(ii)}]  For any $K$-basis $f_1,\ldots,f_{d}$ of $V_1(L)$, we have
\[
\bigcap_{j=1}^{d} \ker \,\varphi_{f_j}
= L^*\cdot K[t].
\]
\end{enumerate}
\end{lemma}
\begin{proof}
${\rm(i)}$ Since $L$ satisfies the property $(P)$ together with Lemma~\ref{cond dim>0}, the $K$-linear map:
\begin{align} \label{isom}
K^d\longrightarrow \mathcal{V}_1(L); \ (c_0,\ldots,c_{d-1}) \mapsto (c_k)_k,
\end{align}
where $c_{k+d}$ for $k\ge 0$ are determined by
\[
c_{d+k}=
-\dfrac{C_k(c_0,\ldots,c_{d+k-1})}
{\sum_{\substack{0\le j \le m\\ m_j-j=d}}a_{m_j,j}(k+d+1)_{m_j}},
\]
where 
\[
C_k(c_0,\ldots,c_{d+k-1})=\sum_{\substack{0\le j \le m \\ m_j-j<d}}a_{m_j,j}(k+m_j-j+1)_jc_{k+m_j-j}+\sum_{j=0}^{m_j}\sum_{i=\max\{0,j-k\}}^{m_j-1}a_{i,j}(k+i-j+1)_jc_{k+i-j}
\]
is a $K$-isomorphism. This shows that $\dim_KV_1(L)=d$.

\medskip

${\rm(ii)}$ Denote by
\[
W = \bigcap_{j=1}^{d} \ker \,\varphi_{f_j}
\]
the corresponding $K$-vector space.
Since $L \cdot f_j \in K[z]$ for all $j$, Corollary~\ref{corollary inc} implies that
\[
L^{*}\cdot K[t] \subseteq W.
\]
We now prove the reverse inclusion.
Let $P(t)\in W$. By the equivalence condition ${\rm(ii)}$ in Lemma~\ref{equiv of (P)}, there exists a polynomial
\(\widetilde{P}(t)\in L^{*}\cdot K[t]\) such that
$
P(t)-\widetilde{P}(t)\in K[t]_{\le d-1}
$.
Hence, it suffices to establish the equality
\begin{equation} \label{determine ker}
W \cap K[t]_{\le d-1}=\{0\}.
\end{equation}

Write
$
f_j(z)=\sum_{k=0}^{\infty}f_{j,k}/z^{k+1}
$ for $1\le j \le d$,
and define the matrix
\[
\mathcal{M}_0 :=
\begin{pmatrix}
f_{1,0} & \cdots & f_{1,d-1}\\
\vdots  & \ddots & \vdots\\
f_{d,0} & \cdots & f_{d,d-1}
\end{pmatrix}
\in \Mat_d(K).
\]
Since each Laurent series $f_j$ is uniquely determined by the coefficients
$(f_{j,k})_{0\le k \le d-1}$ (cf.~the $K$-isomorphism~\eqref{isom}),
and since the Laurent series $\{f_j(z)\}_{1\le j \le d}$ are linearly independent over $K$,
it follows that $\det \,\mathcal{M}_0 \neq 0$.
Suppose that there exists a nonzero polynomial
\[
P(t)=\sum_{j=0}^{d-1} p_j t^j \in W \cap K[t]_{\le d-1}.
\]
By the linearity of the maps $\varphi_{f_j}$ and the assumption $P(t)\in W$,
the nonzero vector $\boldsymbol{p}:={}^{t}(p_0,\ldots,p_{d-1})$ satisfies
$
\mathcal{M}_0 \cdot \boldsymbol{p}=\boldsymbol{0}.
$
This contradicts the invertibility of $\mathcal{M}_0$.
Therefore, \eqref{determine ker} holds, and the proof is complete.
\end{proof}

\section{Linear independence of Pad\'{e}-type approximants} \label{section: linear independence pade}
We keep the notation in Section~\ref{rodrigues formula}. 
In this section, we consider the following situation. 
Let $L\in K[z,\partial_z]$.  Put $d=\ord_{(1,-1)}(L)$. Assume $d\ge 1$ and $L$ satisfies property $(P)$ (see Definition~\ref{property P}).
By Lemma~\ref{property P implies} ${\rm(i)}$, we have $\dim_KV_1(L)=d$. 
Let us take $f_1,\ldots,f_{d}$ a $K$-basis of $V_1(L)$.
For $n\in \N$, we take $R_n\in I_n(L)\setminus\{0\}$.
For a nonnegative integer $\ell$, we define polynomials:
\[
P_{n,\ell}(z)={\rm{Eval}}_{t=z}\left(R^{*}_n\cdot t^{\ell}\right),  \ \ Q_{n,j,\ell}(z)=\varphi_{f_j}\left(\dfrac{P_{n,\ell}(z)-P_{n,\ell}(t)}{z-t}\right) \ \ (1\le j \le d).
\] 
We denote the Laurent series
\[
\mathfrak{R}_{n,j,\ell}(z)=P_{n,\ell}(z)f_j(z)-Q_{n,j,\ell}(z) \ \ (1\le j \le d).
\]
Note, by Proposition~\ref{Pade} {\rm(ii)}, when $P_{n,\ell}(z)\neq 0$, the vector of polynomials $(P_{n,\ell},Q_{n,1,\ell},\ldots,Q_{n,d,\ell})$ forms a weight $n$ Pad\'{e}-type approximant of $(f_1,\ldots,f_d)$ and $(\mathfrak{R}_{n,j,\ell}(z))_j$ is a Pad\'{e}-type approximation of $(f_1,\ldots,f_d)$.  Thus, in any case we have $\ord_{\infty}\,\mathfrak{R}_{n,j,\ell}\ge n+1$ and, by Proposition~\ref{Pade} {\rm(iii)}, we have the expansion:
\begin{align} \label{expand R}
\mathfrak{R}_{n,j,\ell}(z)=\sum_{k=n}^{\infty}\dfrac{\varphi_{f_j}(t^kP_{n,\ell}(t))}{z^{k+1}}.
\end{align}
Define the determinant $\Delta_n(z)$ of $(d+1)\times (d+1)$ matrix by
\[
\Delta_n(z):=\det
\begin{pmatrix}
  P_{n,0}(z) & P_{n,1}(z) & \cdots & P_{n,d}(z)\\
  Q_{n,1,0}(z) & Q_{n,1,1}(z) & \cdots & Q_{n,1,d}(z)\\
  \vdots & \vdots & \ddots & \vdots\\
  Q_{n,d,0}(z) & Q_{n,d,1}(z) & \cdots & Q_{n,d,d}(z)
\end{pmatrix}.
\]
To prove the nonvanishing of $\Delta_n(z)$ is an important task in transcendental number theory (cf. \cite{Siegel}). 
We study a condition so that $\Delta_n(z)\neq 0$ for all $n\in\N$.
The main statement of this section is as follows:
\begin{proposition} \label{equivalence nonzero det} 
We keep the notation as above. Assume that 
\begin{enumerate}
      \item[{\rm{\rm(i)}}] $\ord_{(1,-1)}(R_n)=dn$ for all $n\in \N$.
      \item[{\rm(ii)}] $R_n$ satisfies property $(P)$ for all $n\in \N$.
\end{enumerate}
Then the following are equivalent.
      \begin{enumerate}
      \item[{\rm(a)}] $1,f_1,\ldots,f_{d}$ are linearly independent over $K(z)$.
      \item[{\rm(b)}] $\Delta_n(z)\in K^{\times}$ for all $n\in \N$.
      \end{enumerate}
\end{proposition}
To prove Proposition~\ref{equivalence nonzero det}, we set 
\[
\Theta_n=\det
\begin{pmatrix}
\varphi_{f_1}(t^n R_n^*\cdot 1) & \cdots & \varphi_{f_1}(t^n R_n^*\cdot t^{d-1})\\
\vdots & \ddots & \vdots\\
\varphi_{f_{d}}(t^n R_n^*\cdot 1) & \cdots & \varphi_{f_{d}}(t^n R_n^*\cdot t^{d-1})
\end{pmatrix}\in K,
\]
and prepare the following lemma.
\begin{lemma} \label{const}
We keep the notation as above. Assume that 
\begin{enumerate}
      \item[{\rm{\rm(i)}}] $\ord_{(1,-1)}(R_n)=dn$ for all $n\in \N$.
      \item[{\rm(ii)}] $R_n$ satisfies property $(P)$ for all $n\in \N$.
\end{enumerate}
Then, there exists a nonzero constant $c\in K$ such that
$
\Delta_n(z)=c\cdot \Theta_n.
$
In particular, we have $\Delta_n(z)\in K$.
\end{lemma}
\begin{proof}
Note that the assumption ${\rm(i)}$ and ${\rm(ii)}$ imply, the polynomial $P_{n,\ell}$ satisfies 
\begin{align} \label{deg Pl}
\deg\,P_{n,\ell}=dn+\ell.
\end{align}
For the matrix in the definition of $\Delta_n(z)$, adding $-f_{j}(z)$ times first row to $j+1$ th row for each $1\le j \le d$, 
                     $$ 
                     \Delta_n(z)=(-1)^{d}{\rm{det}}
                     {\begin{pmatrix}
                     P_{n,0}(z) & \dots &P_{n,d}(z)\\
                     \mathfrak{R}_{n,1,0}(z) & \dots & \mathfrak{R}_{n,d,d}(z)\\
                     \vdots    & \ddots & \vdots  \\
                     \mathfrak{R}_{n,d,0}(z) & \dots & \mathfrak{R}_{n,d,d}(z)
                     \end{pmatrix}}. 
                     $$ 
We denote the $(s,t)$th cofactor of the matrix in the right hand side of above equality by $\Delta_{s,t}(z)$.
Then we have, developing along the first row 
\begin{align} \label{formal power series rep delta}
\Delta_n(z)=(-1)^{d}\left(\sum_{\ell=0}^{{{d}}}P_{n,\ell}(z)\Delta_{1,\ell+1}(z)\right).
\end{align}  
The property of the Pad\'{e} approximation ${\rm{ord}}_{\infty}\, {{\mathfrak{R}}}_{n,j,\ell}(z)\ge n+1$ for $1\le j \le d$, $0 \le \ell \le d$ implies
$$
{\rm{ord}}_{\infty}\,\Delta_{1,\ell+1}(z)\ge d(n+1) \ \ \text{for} \ \ 0\le \ell \le d.
$$
Combining equation~\eqref{deg Pl} and above inequality yields
$$P_{n,\ell}(z)\Delta_{1,\ell+1}(z)\in (1/z)\cdot K[[1/z]] \ \ \text{for} \ \ 0\le \ell \le d-1,$$
and 
$$P_{n,d}(z)\Delta_{1,{{d+1}}}(z)\in K[[1/z]].$$
Note that in above relation, the constant term of $P_{n,d}(z)\Delta_{1,{{d+1}}}(z)$ is 
\begin{equation} \label{what const} 
``\text{Coefficient of} \ z^{d(n+1)} \ \text{of} \ P_{n,d}(z)'' \cdot ``\text{Coefficient of} \ 1/z^{d(n+1)} \ \text{of} \ \Delta_{1,{{d+1}}}(z)''.
\end{equation}  
Equation $(\ref{formal power series rep delta})$ implies $\Delta_n(z)$ is a polynomial in $z$ with nonpositive valuation with respect to ${\rm{ord}}_{\infty}$. 
Thus, it has to be a constant.
Finally, by equation~\eqref{expand R}, the coefficient of $1/z^{d(n+1)}$ of  $\Delta_{1,{{d+1}}}(z)$ is $\Theta_n$. 
Combining Equations $(\ref{formal power series rep delta})$, $(\ref{what const})$ and above equality yields 
\[
\Delta_n(z)=
\dfrac{(-1)^d}{(dn)!}\partial^{dn}_z\cdot P_{n,d}(z)\times \Theta_n.
\]
By \eqref{deg Pl}, the constant $c=\tfrac{(-1)^d}{(dn)!}\partial^{dn}_z\cdot P_{n,d}(z)$ is nonzero . This completes the proof of Lemma~\ref{const}.
\end{proof}
\begin{proof}[\textbf{Proof of Proposition~\ref{equivalence nonzero det}}]

\medskip
\noindent
$(a)\Rightarrow (b)$.  
Assume $(a)$ holds.  
Let $n\ge 1$.  
By Lemma~\ref{const}, it suffices to show that $\Theta_n\ne 0$.
By definition we have
\begin{equation}\label{inc} 
 V_n(L)\subseteq V_1(R_n).
\end{equation}
By Lemma~\ref{lem:linear_indep_equiv}, assumption ${\rm(i)}$ yields
\[
  V_n(L)
  = \mathrm{Span}_K\{\pi(z^{k}f_j)\mid 0\le k\le n-1,\ 1\le j\le d\},
  \qquad 
  \dim_K V_n(L)=dn.
\]
Moreover, using Lemma~\ref{property P implies} ${\rm(i)}$, assumptions ${\rm(i)}$ and ${\rm(ii)}$ imply
$
  \dim_K V_1(R_n)=dn
$.
Together with \eqref{inc}, this gives $V_n(L)=V_1(R_n)$.  
Lemma~\ref{property P implies} ${\rm(ii)}$ then yields, for any positive integer $n$,
\begin{equation}\label{important}
  \bigcap_{j=1}^{d}\ \bigcap_{k=0}^{n-1} \ker\, \varphi_{\pi(z^{k}f_j)}
  \;=\; R_n^{*}\cdot K[t].
\end{equation}
Suppose now that $\Theta_n=0$.  
Then there exists a nonzero vector
$
  \boldsymbol{a}={}^t(a_0,\ldots,a_{d-1})\in K^d
$
such that
\[
\begin{pmatrix}
\varphi_{f_1}(t^{n}R_n^{*}\cdot 1) & \cdots & \varphi_{f_1}(t^{n}R_n^{*}\cdot t^{d-1})\\
\vdots & \ddots & \vdots\\
\varphi_{f_d}(t^{n}R_n^{*}\cdot 1) & \cdots & \varphi_{f_d}(t^{n}R_n^{*}\cdot t^{d-1})
\end{pmatrix}
\!\cdot\! \boldsymbol{a}
= \boldsymbol{0}.
\]
Set $Q(t)=\sum_{j=0}^{d-1} a_j t^j$.  
Since $\varphi_{f_j}\circ t^n=\varphi_{\pi(z^{n}f_j)}$, linearity gives
\[
  R_n^{*}\cdot Q(t)\in 
  \bigcap_{j=1}^{d} \ker\, \varphi_{\pi(z^{n}f_j)}.
\]
Using \eqref{important} for both $n$ and $n+1$, we obtain
\[
  R_n^{*}\cdot Q(t)\in
  \left(\bigcap_{j=1}^{d}\ \bigcap_{k=0}^{n-1}\ker\,\varphi_{\pi(z^{k}f_j)}\right)
  \cap
  \left(\bigcap_{j=1}^{d}\ker\,\varphi_{\pi(z^{n}f_j)}\right)
  = R_{n+1}^{*}\cdot K[t].
\]
Hence there exists $P(t)\in K[t]$ such that
$
  R_n^{*}\cdot Q(t)=R_{n+1}^{*}\cdot P(t)
$.
By assumptions ${\rm(i)}$, ${\rm(ii)}$ and Lemma~\ref{equiv of (P)},
\[
  \deg\, R_n^{*}\cdot Q
  = dn + \deg\,Q
  \le d(n+1)-1,
\]
whereas
\[
  \deg\,R_{n+1}^{*}\cdot P
  = d(n+1)+\deg\,P
  \ge d(n+1),
\]
a contradiction.  
Thus $\Theta_n\ne 0$, proving $(b)$.

\bigskip
\noindent
$(b)\Rightarrow (a)$.  
Assume $(b)$ holds and suppose that $1,f_1,\ldots,f_d$ are linearly dependent over $K(z)$.  
Then there exist a positive integer $n$ and polynomials
\[
  (P_0,P_1,\ldots,P_d)\in K[z]^{d+1}\setminus\{\boldsymbol{0}\},
  \qquad
  \max_{1\le j\le d}\{\deg P_j\}=n+1,
\]
such that
\begin{equation*} 
  P_0(z)+\sum_{j=1}^{d} P_j(z)f_j(z)=0.
\end{equation*}
Write $P_j(z)=\sum_{k=0}^{n+1} p_{j,k}z^{k}$.  
From the above identity,
\begin{equation}\label{in Vn(L)}
  \pi\left(\sum_{j=1}^{d} p_{j,n+1} z^{n+1}f_j\right)
  =
  \sum_{j=1}^{d}\sum_{k=0}^{n} p_{j,k}\,\pi(z^{k}f_j)
  \in V_{n+1}(L).
\end{equation}
Put $\boldsymbol{p}=(p_{1,n+1},\ldots,p_{d,n+1})\in K^{d}\setminus\{\boldsymbol{0}\}$.
Then
\[
\boldsymbol{p}\cdot
\begin{pmatrix}
\varphi_{f_1}(t^{n+1}R_{n+1}^{*}\cdot 1) & \cdots & \varphi_{f_1}(t^{n+1}R_{n+1}^{*}\cdot t^{d-1})\\
\vdots & \ddots & \vdots\\
\varphi_{f_d}(t^{n+1}R_{n+1}^{*}\cdot 1) & \cdots & \varphi_{f_d}(t^{n+1}R_{n+1}^{*}\cdot t^{d-1})
\end{pmatrix}
=
\left(
\sum_{j=1}^{d} p_{j,n+1}\,\varphi_{f_j}(t^{n+1}R_{n+1}^{*}\cdot t^{\ell})
\right)_{0\le \ell\le d-1}.
\]
By the $K$-isomorphism $\Phi$ from \eqref{Phi} and Proposition~\ref{key prop},
\[
  \sum_{j=1}^{d} p_{j,n+1}\,\varphi_{f_j}(t^{n+1}R_{n+1}^{*}\cdot t^{\ell})
  = \varphi_{\pi\left(R_{n+1}\cdot \sum_{j=1}^{d} p_{j,n+1} z^{n+1}f_j\right)}(t^{\ell})
  = 0.
\]
The last equality follows from \eqref{in Vn(L)} together with
\begin{align*}
  \pi\left(R_{n+1}\cdot \sum_{j=1}^{d} p_{j,n+1} z^{n+1}f_j\right)
  &= \pi\left(
      R_{n+1}\cdot
      \pi\left(\sum_{j=1}^{d} p_{j,n+1} z^{n+1}f_j\right)
    \right)\\
  &= \pi\left(
      R_{n+1}\cdot
      \sum_{j=1}^{d}\sum_{k=0}^{n} p_{j,k}\,\pi(z^{k}f_j)\right)=0,
\end{align*}
since the morphism $\pi\circ R_{n+1}$ annihilates the element of $V_{n+1}(L)$.
Thus the matrix defining $\Theta_{n+1}$ has a nontrivial kernel, contradicting $\Theta_{n+1}\ne 0$ by assumption $(b)$.  
Hence $1,f_1,\ldots,f_d$ must be linearly independent over $K(z)$.
\end{proof}

\section{Pad\'{e}-type approximants for multiple polylogarithms} \label{section: Pade multiple polylog}
Keeping the notation from Section \ref{rodrigues formula}, we fix positive integers $m,r$ and a field $K$ of characteristic $0$.
Let us fix $\alpha_1,\ldots,\alpha_m\in K^{\times}$ which are pairwise distinct. 
Denote by $\mathcal{S}_r$ the set of indices
\[
\mathcal{S}_r=\{(\boldsymbol{s},\boldsymbol{a})\in \cup_{k=1}^r \left(\N^k\times \{\alpha_1,\ldots,\alpha_m\}^k\right) \, ; \, |\boldsymbol{s}|\le r\}
\] 
and set 
\[
M_r=\#\mathcal{S}_r=(m+1)^r-1,
\]
where no confusion may arise, we write $\mathcal{S}_r = \mathcal{S}$ and $M_r = M$ for simplicity.

\medskip

For $(\boldsymbol{s},\boldsymbol{a})\in \mathcal{S}$ with $\boldsymbol{a}=(\alpha_{i_1},\ldots,\alpha_{i_k})$, we denote the following multiple polylogarithm
\[
f_{\boldsymbol{s},\boldsymbol{a}}(z):=\L_{\boldsymbol{s}}(\alpha_{i_1}/\alpha_{i_2},\ldots,\alpha_{i_k}/z).
\]
For a positive integer $N$, we define the differential operator 
\[
L_N = \frac{1}{N!} z^N \prod_{i=1}^m (z-\alpha_i)^N \partial^N_z \in K[z, \partial_z].
\]
One easily verifies that $\ord_{(1,-1)}(L_N) = mN$ and that $L_N$ satisfies property $(P)$ (see Definition~\ref{property P}). 
Let $r$ be a positive integer. We define the differential operator $L$ by
\[
L = L_{(m+1)^{r-1}} L_{(m+1)^{r-2}} \dots L_{m+1} L_1.
\]
The aim of section is to construct the Pad\'{e}-type approximants of the multiple polylogarithms $(f_{\boldsymbol{s},\boldsymbol{a}}(z))_{(\boldsymbol{s},\boldsymbol{a})\in \mathcal{S}}$.

\medskip

The following is the crucial functional properties of the multiple polylogarithms $(f_{\boldsymbol{s},\boldsymbol{a}}(z))_{(\boldsymbol{s},\boldsymbol{a})\in \mathcal{S}}$.
\begin{proposition} \label{prop:linear_independence_Kz}
Let $n$ be a positive integer.  Then the following properties hold.
\begin{enumerate}
      \item[{\rm{\rm(i)}}] $V_n(L)={\rm{Span}}_K\{\pi(z^kf_{\boldsymbol{s},\boldsymbol{a}})\mid 0\le k \le n-1, (\boldsymbol{s},\boldsymbol{a})\in \mathcal{S}_r\} $.
      \item[{\rm{(ii)}}] The functions $(f_{\boldsymbol{s},\boldsymbol{a}}(z))_{(\boldsymbol{s},\boldsymbol{a}) \in \mathcal{S}}$ together with $1$ are linearly independent over $K(z)$.
      \item[{\rm(iii)}] The differential operator
      \[
      R_n = L_{(m+1)^{r-1}n} L_{(m+1)^{r-2}n} \dots L_{(m+1)n} L_n \in K[z, \partial_z]
      \]
      belongs to the left ideal $I_n(L)$ $($see Definition~$\ref{Rodrigues ideal})$.
\end{enumerate}
\end{proposition}

In order to prove Proposition~\ref{prop:linear_independence_Kz}, we first establish several auxiliary results.
\begin{lemma} \label{differential multiple polylogarithm}
Let $n \in \N$ and let $k$ be an integer with $0 \le k \le n-1$. 
For any $(\bm{s}, \bm{a}) \in \mathcal{S}_r$, we have
\[
L_n \cdot z^k f_{\bm{s}, \bm{a}} \in K[z] +\sum_{(\bm{s}', \bm{a}') \in \mathcal{S}_{r-1}} K[z]_{\le (m+1)n-1} f_{\bm{s}', \bm{a}'}(z).
\]
\end{lemma}
\begin{proof}
Let $(\bm{s}, \bm{a}) = (s_1, \dots, s_k, \alpha_{i_1}, \dots, \alpha_{i_k})$. 
A straightforward computation yields
\[
\partial_z \cdot f_{\bm{s}, \bm{a}} =
\begin{cases}
-\dfrac{1}{z} f_{\bm{s}', \bm{a}} & \text{if } s_k > 1, \\
-\dfrac{\alpha_{i_k}}{z(z-\alpha_{i_k})} f_{\bm{s}', \bm{a}'} & \text{if } s_k = 1,
\end{cases}
\]
where $\bm{s}' = (s_1, \dots, s_k-1)$ if $s_k > 1$, and $\bm{s}' = (s_1, \dots, s_{k-1})$, $\bm{a}' = (\alpha_{i_1}, \dots, \alpha_{i_{k-1}})$ if $s_k = 1$. The desired relation then follows from the Leibniz rule.
\end{proof}

For $n, r \in \mathbb{N}$, we define the $K$-vector space
\[
V_{n,r} = \mathrm{Span}_K \{ \pi(z^k f_{\bm{s}, \bm{a}}) \mid 0 \le k \le n-1, ({\bm{s}, \bm{a}}) \in \mathcal{S}_r \}.
\]
Note that $L_n$ belongs to the left ideal $I_n(L_1)$ for any $n \in \mathbb{N}$ (cf. Example~\ref{order 1} with $a_1(z) = z \prod_{i=1}^m (z - \alpha_i)$ and $a_0(z) = 0$). By applying Lemma~\ref{differential multiple polylogarithm} repeatedly, we obtain
\begin{align*}
L \cdot f_{\bm{s}, \bm{a}} &= L_{(m+1)^{r-1}} (L_{(m+1)^{r-2}} \dots L_{m+1} L_1 \cdot f_{\bm{s}, \bm{a}}) \\
&\in L_{(m+1)^{r-1}} \left( K[z] + \sum_{i=1}^m K[z]_{\le (m+1)^{r-1}-1} f_{1, \alpha_i} \right) \subset K[z].
\end{align*}
This implies $V_{1,r} \subseteq V_1(L)$ and, more generally,
\begin{equation} \label{Vnr subset Vn}
V_{n,r} \subseteq V_n(L) \quad \text{for any } n \in \mathbb{N}.
\end{equation}

Lemma~\ref{differential multiple polylogarithm} ensures that the following $K$-homomorphism is well-defined:
\[
\pi \circ L_n : V_{n,r} \longrightarrow V_{(m+1)n,r-1}.
\]

\begin{lemma} \label{surjective}
For any $n,r\in \N$ with $r\ge2$, the morphism $\pi\circ L_n:V_{n,r}\rightarrow V_{(m+1)n,r-1}$ is surjective.
\end{lemma}
\begin{proof}
We proceed the proof by induction on $r$. 
Let $a(z)=z\prod_{i=1}^m(z-\alpha_i)$. 
First, consider the base case $r=2$ for any  $r=2$ and any $n\in \N$. 
Fix $1\le i \le m$. 
Lemma~\ref{differential multiple polylogarithm} and the Leibniz rule imply for $0\le \ell \le n-1$ and $1\le k \le m$:
\begin{align*}
\pi\circ L_n(z^\ell f_{2,\alpha_i}) &= \dfrac{a^n}{n!}\sum_{j=0}^{\ell}\binom{n}{j}(\ell-j+1)_j\left(\dfrac{-1}{z}\right)^{(n-j-1)}f_{1,\alpha_i} \mod K[z],\\
\pi\circ L_n(z^\ell f_{(1,1),(\alpha_{i},\alpha_k)}) &= \dfrac{a^n}{n!}\sum_{j=0}^{\ell}\binom{n}{j}(\ell-j+1)_j\left(\dfrac{-\alpha_k}{z(z-\alpha_k)}\right)^{(n-j-1)}f_{1,\alpha_i} \mod K[z],
\end{align*}
where $(a)_j$ is Pochhammer symbol, given by $(a)_0=1$ and $(a)_j=a(a+1)\cdots(a+j-1)$ for $j\ge 1$, and $(-1/z)^{(j)}=\partial^j_z(-1/z)$. 
Since the $K$-space $V_{(m+1)n,1}$ is spanned by the set 
\[
\{\pi(z^kf_{1,\alpha_i}) \mid 0\le k \le (m+1)n-1, 1\le i \le m\},
\]
the surjectivity $\pi\circ L_n:V_{n,2}\rightarrow V_{(m+1)n,1}$ is equivalent to
      \begin{enumerate}
      \item[(*)] The $K$-space $K[z]_{\le (m+1)n-1}$ is spanned by the $(m+1)n$ polynomials 
                               \begin{align*}
&a^n\sum_{j=0}^{\ell}\binom{n}{j}(\ell-j+1)_j\left(\dfrac{-1}{z}\right)^{(n-j-1)}, \\
&a^n\sum_{j=0}^{\ell}\binom{n}{j}(\ell-j+1)_j\left(\dfrac{-\alpha_k}{z(z-\alpha_k)}\right)^{(n-j-1)}
\end{align*}
for $0\le \ell \le n-1$ and $1\le k \le m$.
      \end{enumerate}
Taking into account the identity $\tfrac{-\alpha_k}{z(z-\alpha_k)}=\tfrac{1}{z}-\tfrac{1}{z-\alpha_k}$ and the fact that $\alpha_k$ are pairwise distinct, the partial fraction expansion yields:
\[
\biggl\{\dfrac{P(z)}{a^n(z)} \Bigm\vert  P(z)\in K[z]_{\le (m+1)n-1}\biggr\}={\rm{Span}}_K\biggl\{\left(\dfrac{-1}{z}\right)^{(\ell)}, \left(\dfrac{-\alpha_k}{z(z-\alpha_k)}\right)^{(\ell)} \Bigm\vert 0\le \ell \le n-1, 1\le k \le m\biggr\}.
\] 
This confirms that assertion $(*)$ holds, establishing the base case.

\medskip

Now, assume the statement holds for some $r \ge 2$ and all $n \in \mathbb{N}$. 
For $r+1$, the inductive hypothesis implies it is sufficient to prove the surjectivity of the induced morphism:
\begin{align} \label{final surj}
\pi\circ L_n:V_{n,r+1}/V_{n,r}\longrightarrow V_{(m+1)n,r}/V_{(m+1)n,r-1}.
\end{align}
Fix $(\boldsymbol{s},\boldsymbol{a})=(s_1,\ldots,s_k,\alpha_{i_1},\ldots,\alpha_{i_k})\in \mathcal{S}_r$. 
By Lemma~\ref{differential multiple polylogarithm} and the Leibniz rule, for $0\le \ell \le n-1$,
\begin{align*}
\pi\circ L_n(z^\ell f_{(s_1,\ldots,s_k+1),\boldsymbol{a}}) &\equiv \dfrac{a^n}{n!}\sum_{j=0}^{\ell}\binom{n}{j}(\ell-j+1)_j\left(\dfrac{-1}{z}\right)^{(n-j-1)}f_{\boldsymbol{s},\boldsymbol{a}}  \mod V_{(m+1)n,r},\\
\pi\circ L_n(z^\ell f_{(\boldsymbol{s},1),(\boldsymbol{a},\alpha_{k})}) &\equiv \dfrac{a^n}{n!}\sum_{j=0}^{\ell}\binom{n}{j}(\ell-j+1)_j\left(\dfrac{-\alpha_k}{z(z-\alpha_k)}\right)^{(n-j-1)}f_{\boldsymbol{s},\boldsymbol{a}} \mod V_{(m+1)n,r}.
\end{align*}
Since the space $V_{(m+1)n,r}/V_{(m+1)n,r-1}$ is spanned by the set
\[
\Bigl\{\pi(z^kf_{\boldsymbol{s},\boldsymbol{a}}) \mod V_{(m+1)n,r-1} \mid 0\le k \le (m+1)n-1, ({\boldsymbol{s},\boldsymbol{a}}) \in \mathcal{S}_{r}\Bigr\},
\]
the same argument used in the $r=2$ case ensures the surjectivity of \eqref{final surj}. 
This completes the induction. 
\end{proof} 

\begin{proof}[\textbf{Proof of Proposition~\ref{prop:linear_independence_Kz}}]
{\rm{(i)}} By the definition of $V_n(L)$, we have $\dim_K V_n(L) \le nM_r$. In view of \eqref{Vnr subset Vn}, it suffices to show that 
\begin{equation} \label{dim Vnr}
\dim_K V_{n,r} = nM_r \quad \text{for all } n, r \in \mathbb{N}.
\end{equation}
We proceed by induction on $r$. 
For $r=1$, since the $\alpha_i$ are pairwise distinct, the functions $f_{1,\alpha_i}(z) = \mathrm{Li}_1(\alpha_i/z)$ for $1 \le i \le m$ are linearly independent over $K(z)$ (cf. \cite{Vaa}). Thus, \eqref{dim Vnr} holds for any $n \in \mathbb{N}$.

Next, assume that \eqref{dim Vnr} holds for $r$. For the case $r+1$, Lemma~\ref{surjective} implies the following chain of inequalities:
\begin{align*}
nM_{r+1} \ge \dim_K V_{n,r+1} &= \dim_K (\mathrm{Im}(\pi \circ L_n)) + \dim_K (\ker(\pi \circ L_n)) \\
&\ge \dim_K V_{(m+1)n,r} + \dim_K V_{n,1} \\
&= (m+1)nM_r + mn = nM_{r+1},
\end{align*}
where the last inequality follows from the inductive hypothesis and the fact that $V_{n,1} \subset \ker(\pi \circ L_n)$. This forces $\dim_K V_{n,r+1} = nM_{r+1}$, completing the induction and the proof of (i).

\medskip

{\rm{(ii)}} This follows immediately from Lemma~\ref{lem:linear_indep_equiv} and the dimension formula established in (i).

\medskip

{\rm{(iii)}} By item (i) and the definition of the ideal $I_n(L)$, it is sufficient to show that 
\[
R_n \cdot z^k f_{\bm{s}, \bm{a}} \in K[z] \quad \text{for } 0 \le k \le n-1 \text{ and } (\bm{s}, \bm{a}) \in \mathcal{S}.
\]
Applying Lemma~\ref{differential multiple polylogarithm} repeatedly, we find that
\[
L_{(m+1)^{r-2}n} \dots L_{(m+1)n} L_n \cdot z^k f_{\bm{s}, \bm{a}} \in K[z] \oplus \bigoplus_{i=1}^m K[z]_{\le (m+1)^{r-1}n-1} f_{1, \alpha_i}.
\]
Recalling the case $r=1$ again, we observe that $L_{(m+1)^{r-1}n}$ maps the right-hand side into $K[z]$. Therefore,
\[
R_n \cdot z^k f_{\bm{s}, \bm{a}} \in L_{(m+1)^{r-1}n} \left( K[z] \oplus \bigoplus_{i=1}^m K[z]_{\le (m+1)^{r-1}n-1} f_{1, \alpha_i} \right) \subset K[z],
\]
which completes the proof.
\end{proof}

\subsection{Pad\'{e}-type approximants of multiple polylogarithms}
The next result is a direct consequence of Proposition~\ref{prop:linear_independence_Kz} and Proposition~\ref{Pade}.
In what follows, for the $K$-isomorphism $\Phi$ defined in \eqref{Phi}, we set
\begin{align} \label{varphi s a}
\varphi_{\boldsymbol{s},\boldsymbol{a}}:=\Phi\bigl(f_{\boldsymbol{s},\boldsymbol{a}}\bigr) \ \text{for} \  (\boldsymbol{s},\boldsymbol{a})\in \mathcal{S}.
\end{align}
\begin{corollary}\label{pade multiple polylog}
Let $n,\ell$ be nonnegative integers. Define the polynomials
\begin{align*}
P_{n,\ell}(z) = {\rm Eval}_{t=z}(R^{*}_n\cdot t^{\ell}), \ \ Q_{n,\boldsymbol{s},\boldsymbol{a},\ell}(z)= \varphi_{\boldsymbol{s},\boldsymbol{a}}\!\left(\frac{P_{n,\ell}(z)-P_{n,\ell}(t)}{z-t}\right) \ \text{for} \ (\boldsymbol{s},\boldsymbol{a})\in \mathcal{S}.
\end{align*}
Then the vector
$
\bigl(P_{n,\ell},\,Q_{n,\boldsymbol{s},\boldsymbol{a},\ell}\bigr)_{(\boldsymbol{s},\boldsymbol{a})\in \mathcal{S}}
$
forms a weight $n$ Pad\'{e}-type approximant of 
$
\bigl(f_{\boldsymbol{s},\boldsymbol{a}}(z))_{(\boldsymbol{s},\boldsymbol{a})\in \mathcal{S}}.
$
\end{corollary}
\begin{proof}
By definition of $R_n$, we have 
\begin{align} \label{deg R^*_nP}
\deg \, R^{*}_n\cdot t^\ell=Mn+\ell
\end{align}
and, in particular the polynomial $P_{n,\ell}(z)$ is nonzero. Proposition~\ref{Pade} implies that the statement holds.
\end{proof}
We keep the notation in Corollary~\ref{pade multiple polylog}.
Define the determinant of the $(m+1)^r\times (m+1)^r$ matrix
\[
\Delta_{n}(z)
  := \det
  \begin{pmatrix}
  P_{n,\ell}(z) \\
  Q_{n,\boldsymbol{s},\boldsymbol{a},\ell}(z)
  \end{pmatrix}_{\substack{0\le \ell \le M \\ (\boldsymbol{s},\boldsymbol{a})\in \mathcal{S}}}.
\]
\begin{corollary}\label{multiple polylog det non zero}
We have $\Delta_{n}(z)\in K^{\times}$.
\end{corollary}
\begin{proof}
Equation \eqref{sum of ord} together with the definition of $\ord_{(1,-1)}$ yields
\[
\ord_{(1,-1)}(R_{n}) = Mn = \ord_{(1,-1)}(L)n,
\]
so $R_n$ satisfies property ${\rm(i)}$.  
Next, combining equation~\eqref{deg R^*_nP} with Lemma~\ref{equiv of (P)} implies that $R_n$ satisfies property $(P)$.  
As we showed in Proposition~\ref{prop:linear_independence_Kz}~$({\rm{ii}})$, the $K$-basis $\{f_{\boldsymbol{s},\boldsymbol{a}}\}_{(\boldsymbol{s},\boldsymbol{a})\in \mathcal{S}}$ of $V_1(L)$ together with $1$ are linearly independent over 
$K(z)$. Applying Proposition~\ref{equivalence nonzero det} implies the assertion.
\end{proof}
\begin{remark}
Corollary~\ref{pade multiple polylog} and Corollary~\ref{multiple polylog det non zero} were established by Sorokin \cite[Lemma~2, Lemma~6]{S2} in the case where $\alpha_1,\ldots,\alpha_m$ are positive rational numbers.
\end{remark}

\section{Estimates} \label{section: estimates} 
We keep the notation in Section~\ref{section: Pade multiple polylog}. 
We fix positive integers $m,r$ and an algebraic number field $K$. 
Let $\bm{\alpha} = (\alpha_1, \dots, \alpha_m) \in (K^{\times})^m$ be a vector whose coordinates are pairwise distinct. 
Recall that $M = (m+1)^r - 1$ and the morphism $\varphi_{\boldsymbol{s},\boldsymbol{a}}$ defined in~\eqref{varphi s a} for $(\boldsymbol{s},\boldsymbol{a})\in \mathcal{S}$.  

\medskip

For $0 \le \ell \le M$ and $(\bm{s}, \bm{a}) \in \mathcal{S}$, Corollary~\ref{pade multiple polylog} states that the polynomials
\begin{equation*} 
    P_{n,\ell}(z) = \mathrm{Eval}_{t=z}(R^{*}_n \cdot z^{\ell}), \quad 
    Q_{n,\bm{s},\bm{a},\ell}(z) = \varphi_{\bm{s},\bm{a}} \left( \frac{P_{n,\ell}(z) - P_{n,\ell}(t)}{z-t} \right)
\end{equation*}
form Pad\'{e}-type approximants to $(f_{\bm{s},\bm{a}}(z))_{(\bm{s},\bm{a}) \in \mathcal{S}}$. 

In this section, we describe the asymptotic behavior as $n \to \infty$ of the polynomials $P_{n,\ell}(z)$ and $Q_{n,\bm{s},\bm{a},\ell}(z)$ evaluated at $\beta \in K$, as well as the Pad\'{e}-type approximation 
\begin{equation*}
    P_{n,\ell}(\beta) f_{\bm{s},\bm{a}}(\beta) - Q_{n,\bm{s},\bm{a},\ell}(\beta)
\end{equation*}
for $\beta \in K$ with $|\beta|_v > H_v(\boldsymbol{\alpha})$.

Furthermore, we fix the following notation. Let $v \in \mathfrak{M}_K$. 
Recall $\varepsilon_v = 1$ if $v \mid \infty$ and $\varepsilon_v = 0$ if $v \nmid \infty$. 
For a polynomial $P(z) = \sum_{k=0}^n p_k z^k \in K[z]$, we denote by $\|P\|_v$ the maximum $v$-adic absolute value of its coefficients:
\begin{equation*}
    \|P\|_v := \max_{0 \le k \le n} \{ |p_k|_v \}.
\end{equation*}
For $\beta \in K$, using the triangle inequality when $v \in \mathfrak{M}^{\infty}_K$ (resp. the strong triangle inequality when $v \in \mathfrak{M}^{f}_K$), we obtain
\begin{equation} \label{triangle inequality}
    \log |P(\beta)|_v \le \varepsilon_v \log (\deg P + 1) + \log \|P\|_v + \deg P\cdot h_v(\beta).
\end{equation}

Let $n$ be a positive integer. We denote by $d_n$ the least common multiple of $1, \dots, n$. In the following, we define the differential operator
\begin{equation*}
    \mathcal{L}_n = \frac{1}{n!} \partial^n_z z^n \prod_{i=1}^m (z - \alpha_i)^n \in K[z, \partial_z].
\end{equation*}
Note that 
\[
P_{n,\ell}(z)=(-1)^{\tfrac{M}{m}}\mathcal{L}_n\mathcal{L}_{(m+1)n}\ldots \mathcal{L}_{(m+1)^{r-1}n}\cdot z^{\ell}.
\]
In the following, we refer the symbols $o(1)$ and $o(n)$ refer to the limit as $n \to \infty$.

We begin by establishing the following fundamental lemma, which provides the necessary estimates for the $v$-adic norms of our constructions.

\begin{lemma} \label{lemma:integrality}
Let $v \in \mathfrak{M}_K$, $P \in K[z]$ be a polynomial of degree $N$, and $n$ be a positive integer. The following assertions hold$:$
\begin{enumerate}
    \item[\rm(i)] $\displaystyle \left\| \frac{1}{n!} \partial^n_z \cdot z^n P \right\|_v \le \binom{n+N}{n}^{\frac{d_v}{d}\varepsilon_v} \|P\|_v$.
    \item[\rm(ii)] $\displaystyle \left\| \prod_{i=1}^m (z - \alpha_i)^n \right\|_v \le (n+1)^{m\varepsilon_v} \cdot 2^{mn \frac{d_v}{d}\varepsilon_v} \prod_{i=1}^m H_v(\alpha_i)^n$.
    \item[\rm(iii)] $\displaystyle \| \mathcal{L}_n \cdot P \|_v \le (mn + N + 1)^{(m+1)\varepsilon_v} \left( 2^{mn} \binom{(m+1)n+N}{n} \right)^{\frac{d_v}{d}\varepsilon_v} \prod_{i=1}^m H_v(\alpha_i)^n \|P\|_v$.
    \item[\rm(iv)] For $(\bm{s}, \bm{a}) \in \mathcal{S}$, we have
    \begin{equation*}
        |\varphi_{\bm{s},\bm{a}}(P(t))|_v \le (N+1)^{(r+1)\varepsilon_v} |d_{N+1}^r|_v^{\varepsilon_v-1} H_v(\bm{\alpha})^{N+1} \|P\|_v.
    \end{equation*}
    \item[\rm(v)] For $(\bm{s}, \bm{a}) \in \mathcal{S}$, let $Q(z) = \varphi_{\bm{s},\bm{a}} \left( \frac{P(z)-P(t)}{z-t} \right)$. Then
    \begin{equation*}
        \|Q\|_v \le (N+1)^{(r+1)\varepsilon_v} |d_{N+1}^r|_v^{\varepsilon_v-1} H_v(\bm{\alpha})^{N+1} \|P\|_v.
    \end{equation*}
\end{enumerate}
\end{lemma}

\begin{proof}
Write $P = \sum_{k=0}^{N} p_k z^k$. 

\medskip
\noindent
{\rm(i)} A direct computation yields the identity
\[
\frac{1}{n!} \partial_z^n \cdot z^n P(z) = \sum_{j=0}^{N} \binom{j+n}{n} p_j z^j.
\]
By the definition of the $v$-adic norm, we obtain
\[
\left\| \frac{1}{n!} \partial_z^n \cdot z^n P(z) \right\|_v \le \max_{0 \le j \le N} \left\{ \left| \binom{j+n}{n} \right|_v |p_j|_v \right\} \le \binom{n+N}{n}^{\frac{d_v}{d}\varepsilon_v} \|P\|_v,
\]
which proves {\rm(i)}.

\medskip
\noindent
{\rm(ii)} Expanding the product, we have
\[
\prod_{i=1}^m (z - \alpha_i)^n = \sum_{k=0}^{mn} \left( \sum_{\substack{0 \le k_i \le n \\ \sum k_i = k}} \prod_{i=1}^m \binom{n}{k_i} (-\alpha_i)^{n-k_i} \right) z^k.
\]
Using the inequality $\binom{n}{k} \le 2^n$ and applying the ultrametric inequality (if $v \in \mathfrak{M}^f_K$) or the standard triangle inequality (if $v \in \mathfrak{M}^\infty_K$), we obtain the desired estimate.

\medskip
\noindent
{\rm(iii)} For any $Q \in K[z]$, the $v$-adic norm satisfies
\begin{equation} \label{PQ}
\|PQ\|_v \le (\deg P + \deg Q + 1)^{\varepsilon_v} \|P\|_v \|Q\|_v.
\end{equation}
Applying {\rm(i)} to the polynomial $P \prod_{i=1}^m (z-\alpha_i)^n$, we find
\[
\|\mathcal{L}_n\cdot P\|_v \le \binom{(m+1)n+N}{n}^{\frac{d_v}{d}\varepsilon_v} \left\| P \prod_{i=1}^m (z-\alpha_i)^n \right\|_v.
\]
Combining this with \eqref{PQ} and the estimate from {\rm(ii)} yields the result.

\medskip
\noindent
{\rm(iv)} Let $(\bm{s}, \bm{a}) \in \mathcal{S}$ with $\bm{s} \in \mathbb{N}^k$ and $\bm{a}=(\alpha_{i_1}, \ldots, \alpha_{i_k})$. By the definition of $\varphi_{\bm{s},\bm{a}}$, for $j \in \mathbb{Z}_{\ge 0}$, we have
\[
\varphi_{\bm{s},\bm{a}}(t^j) = 
\begin{cases} 
0 & \text{if } j < k, \\ 
\displaystyle \sum_{0 < n_1 < \dots < n_{k-1} < j+1} \frac{\alpha_{i_1}^{n_1} \cdots \alpha_{i_k}^{j+1-n_{k-1}}}{n_1^{s_1} \cdots n_{k-1}^{s_{k-1}} (j+1)^{s_k}} & \text{if } j \ge k.
\end{cases}
\]
Bounding the coefficients leads to
\[
|\varphi_{\bm{s},\bm{a}}(t^j)|_v \le (j+1)^{r\varepsilon_v} |d_{j+1}^r|_v^{\varepsilon_v-1} H_v(\bm{\alpha})^{j+1},
\]
from which the general estimate for $P(t)$ follows by the triangle inequality.

\medskip
\noindent
{\rm(v)} For $k \ge 1$, we use the identity $(z^k - t^k)/(z-t) = \sum_{u=0}^{k-1} t^{k-1-u} z^u$. Thus,
\[
\frac{P(z)-P(t)}{z-t} = \sum_{u=0}^{N-1} \left( \sum_{k=u+1}^{N} p_k t^{k-1-u} \right) z^u.
\]
Applying $\varphi_{\bm{s},\bm{a}}$ with respect to $t$, we obtain
\[
Q(z) = \sum_{u=0}^{N-1} \left( \sum_{k=u+1}^{N} p_k \varphi_{\bm{s},\bm{a}}(t^{k-1-u}) \right) z^u.
\]
Invoking {\rm(iv)} and the $v$-adic triangle inequality, we conclude
\[
\|Q\|_v = \max_{0 \le u \le N-1} \left| \sum_{k=u+1}^{N} p_k \varphi_{\bm{s},\bm{a}}(t^{k-1-u}) \right|_v \le (N+1)^{(r+1)\varepsilon_v} |d_{N+1}^r|_v^{\varepsilon_v-1} H_v(\bm{\alpha})^{N+1} \|P\|_v.
\]
This completes the proof.
\end{proof}

\subsection{Absolute values of the Pad\'{e}-type approximants} \label{subsection: Abs estimate}
The aim of this subsection is to prove the following proposition. 
\begin{proposition} \label{estimate approximants}
Let $v\in \mathfrak{M}_K$ and $\beta\in K$. Let $n\ge 0$ and $0\le \ell\le M$ be integers.
\begin{align*}
\log\max_{(\boldsymbol{s},\boldsymbol{a})\in \mathcal{S}}\{|P_{n,\ell}(\beta)|_v,|Q_{n,\boldsymbol{s},\boldsymbol{a},\ell}(\beta)|_v\} 
&\le n \biggl(Mh_v(\beta)+\varepsilon_v\dfrac{d_v}{d} \Bigl( M\log 2 + \dfrac{r(r+1)}{2}\log(m+1) + r \Bigr)\\
&+  Mh_v(\boldsymbol{\alpha})+\dfrac{M}{m}\sum_{i=1}^m h_v(\alpha_i)+o(1)\biggr) + (\varepsilon_v-1)\log |d^r_{Mn+M}|_v,
\end{align*}
where $o(1)=0$ for almost all places $v$.
\end{proposition}
\begin{proof}
First we estimate for $P_{n,\ell}$. 
Applying Lemma~\ref{lemma:integrality} $m$-times, we see 
\begin{align} \label{first |Pnl|}
\|P_{n,\ell}\|_v\le e^{o(n)}\prod_{j=0}^{r-1}\left[\left(\binom{(m+1)^{j+1}n+M}{n}2^{m(m+1)^jn}\right)^{\varepsilon_v\tfrac{d_v}{d}}\prod_{i=1}^{m}H_v(\alpha_i)^{(m+1)^jn}\right],
\end{align}
where $o(n)=0$ when $v\in \mathfrak{M}^{f}_K$. Stirling formula for the binomial coefficient $\binom{(m+1)^jn+M}{n}$ implies 
\[
\log \binom{(m+1)^jn+M}{n}\le n\left( \log\left(\dfrac{(m+1)^{j(m+1)^j}}{((m+1)^j-1)^{(m+1)^j-1}}\right) + o(1)\right) \ \ \text{for} \ \ j\ge1.
\]
Using the inequality $x\log(x)-(x-1)\log(x-1)<\log(x)+1$ for $x>1$, we have
\[
\sum_{j=1}^r\log\left(\dfrac{(m+1)^{j(m+1)^j}}{((m+1)^j-1)^{(m+1)^j-1}}\right)\le \sum_{j=1}^r(j\log(m+1)+1)=\dfrac{r(r+1)}{2}\log(m+1)+r.
\]
Combining the above inequality with~\eqref{first |Pnl|} yields
\begin{align} \label{|Pnl|}
\log\,||P_{n,\ell}||_v\le n\left(\varepsilon_v\dfrac{d_v}{d} \Bigl( M\log 2 + \dfrac{r(r+1)}{2}\log(m+1) + r \Bigr)+  \dfrac{M}{m}\sum_{i=1}^m h_v(\alpha_i)+o(1)\right),
\end{align}
where $o(1)=0$ for $v\in \mathfrak{M}^{f}_K$.

\medskip

Combining the equations~\eqref{triangle inequality} and~\eqref{|Pnl|} together with ${\rm{deg}}\,P_{n,\ell}=Mn+\ell$ yields
\begin{align*}
\log|P_{n,\ell}(\beta)|_v\le n\left(\varepsilon_v\dfrac{d_v}{d} \Bigl( M\log 2 + \dfrac{r(r+1)}{2}\log(m+1) + r \Bigr)+  \dfrac{M}{m}\sum_{i=1}^m h_v(\alpha_i)+Mh_v(\beta)+o(1)\right).
\end{align*}

Applying Lemma~\ref{lemma:integrality} ${\rm(v)}$ for $P_{n,\ell}$ together with the estimate~\eqref{|Pnl|}, we have
\begin{align*}
\log ||Q_{n,\boldsymbol{s},\boldsymbol{a},\ell}||_v & \le n(Mh_v(\boldsymbol{\alpha})+o(1)) + (\varepsilon_v-1)\log |d^r_{Mn+M}|_v + \log||P_{n,\ell}||_v \\
                         & \le n \biggl(\varepsilon_v\dfrac{d_v}{d} \Bigl( M\log 2 + \dfrac{r(r+1)}{2}\log(m+1) + r \Bigr)\\
                         &+  Mh_v(\boldsymbol{\alpha})+\dfrac{M}{m}\sum_{i=1}^m h_v(\alpha_i)+o(1)\biggr) + (\varepsilon_v-1)\log |d^r_{Mn+M}|_v.
\end{align*}
Therefore the equation~\eqref{triangle inequality} together with ${\rm{deg}}\,Q_{n,\boldsymbol{s},\boldsymbol{a},\ell}=Mn+\ell-1$ yields 
\begin{align*}
\log |Q_{n,\boldsymbol{s},\boldsymbol{a},\ell}(\beta)|_v& \le n \biggl(Mh_v(\beta)+\varepsilon_v\dfrac{d_v}{d} \Bigl( M\log 2 + \dfrac{r(r+1)}{2}\log(m+1) + r \Bigr)\\
                         &+  Mh_v(\boldsymbol{\alpha})+\dfrac{M}{m}\sum_{i=1}^m h_v(\alpha_i)+o(1)\biggr) + (\varepsilon_v-1)\log |d^r_{Mn+M}|_v
\end{align*}
as desired.
\end{proof}

\subsection{Absolute values of the Pad\'{e}-type approximations}
Let $0\le n$ and $0 \le \ell \le M$ be integers and $(\boldsymbol{s},\boldsymbol{a})\in \mathcal{S}$. 
We keep the notation introduced in Section~\ref{subsection: Abs estimate}. Let $v \in \mathfrak{M}_K$. 
This section is devoted to estimating the $v$-adic absolute values of the Pad\'{e}-type approximations 
\[
\mathfrak{R}_{n,\boldsymbol{s},\boldsymbol{a},\ell}(z) = P_{n,\ell}(z)f_{\boldsymbol{s},\boldsymbol{a}}(z)-Q_{n,\boldsymbol{s},\boldsymbol{a},\ell}(z)
\] 
evaluated at some point $\beta \in K$. 
\begin{proposition} \label{Remainder}
Let $v\in \mathfrak{M}_K$ and $\beta \in K$ with $|\beta|_v > H_v(\boldsymbol{\alpha})$.
Then the Laurent series $\mathfrak{R}_{n,\boldsymbol{s},\boldsymbol{a},\ell}(z)$ converges at $z = \beta$ and satisfies
\[
\begin{aligned}
\log \max_{\substack{(\boldsymbol{s},\boldsymbol{a})\in \mathcal{S} \\ 0\le \ell \le M}} |\mathfrak{R}_{n,\boldsymbol{s},\boldsymbol{a},\ell}(\beta)|_v 
&\le n \biggl( -h_v(\beta) + \dfrac{M}{m}\sum_{i=1}^m h_v(\alpha_i) + (M+1)h_v(\boldsymbol{\alpha}) \\
&\quad + \varepsilon_v \frac{d_v}{d} \Bigl( M \log 2 + \frac{r(r+1)}{2} \log(m+1) + r \Bigr) + o(1) \biggr).
\end{aligned}
\]
\end{proposition} 
\begin{proof} 
Using Proposition~\ref{Pade}~(iii), we have the representation
\[
\mathfrak{R}_{n,\boldsymbol{s},\boldsymbol{a},\ell}(z) = \sum_{k=0}^{\infty} \frac{\varphi_{\boldsymbol{s}, \boldsymbol{\alpha}}(t^{k+n}P_{n,\ell}(t))}{z^{k+n+1}}.
\]
Applying Lemma~\ref{lemma:integrality}~(v) to the polynomial $P(t) = t^{k+n}P_{n,\ell}(t)$, it follows that
\begin{align*}
|\varphi_{\boldsymbol{s},\boldsymbol{a}}(t^{k+n}P_{n,\ell}(t))|_v \le \, & ((M+1)n+M+k+1)^{(r+1)\varepsilon_v} \left|d_{(M+1)n+M+k+1}^r \right|_v^{\varepsilon_v-1} \\
& \quad \cdot H_v(\boldsymbol{\alpha})^{(M+1)n+M+k+1} \|P_{n,\ell}\|_v. 
\end{align*}
By employing the bound $|d_N|_v^{\varepsilon_v-1} \le N^{(\varepsilon_v-1)d_v/d}$ for $N \in \mathbb{N}$, we obtain
\begin{align*}
|\varphi_{\boldsymbol{s},\boldsymbol{a}}(t^{k+n}P_{n,\ell}(t))|_v \le e^{o(n)} H_v(\boldsymbol{\alpha})^{(M+1)n+k+1} (k+1)^{r+1} \|P_{n,\ell}\|_v.
\end{align*}

\medskip

\textbf{Case 1}: $v \in \mathfrak{M}^{\infty}_K$. If $v$ is an Archimedean place, the condition $|\beta|_v > H_v(\boldsymbol{\alpha})$ ensures that the numerical series 
\[
\sum_{k=0}^{\infty} \frac{H_v(\boldsymbol{\alpha})^{k+1}(k+1)^{r+1}}{|\beta|^{k+1}_v}
\]
converges in $K_v$. Consequently, the remainder series $\mathfrak{R}_{n,\boldsymbol{s},\boldsymbol{a},\ell}(\beta)$ converges in $K_v$ and satisfies
\begin{equation} \label{v infty R}
|\mathfrak{R}_{n,\boldsymbol{s},\boldsymbol{a},\ell}(\beta)|_v \le e^{o(n)} |\beta|^{-n}_v H_v(\boldsymbol{\alpha})^{(M+1)n} \|P_{n,\ell}\|_v.
\end{equation}

\medskip

\textbf{Case 2}: $v \in \mathfrak{M}^{f}_K$. If $v$ is a non-Archimedean place, the condition $|\beta|_v > H_v(\boldsymbol{\alpha})$ similarly implies the convergence of $\mathfrak{R}_{n,\boldsymbol{s},\boldsymbol{a},\ell}(\beta)$ in $K_v$. By the strong triangle inequality, we have
\begin{align}
|\mathfrak{R}_{n,\boldsymbol{s},\boldsymbol{a},\ell}(\beta)|_v &\le e^{o(n)} \max_{k \ge 0} \left\{ |\beta|^{-k-n-1}_v H_v(\boldsymbol{\alpha})^{(M+1)n+k+1} (k+1)^{r+1} \|P_{n,\ell}\|_v \right\} \nonumber \\
&= e^{o(n)} |\beta|^{-n}_v H_v(\boldsymbol{\alpha})^{(M+1)n} \|P_{n,\ell}\|_v. \label{v mid infty R} 
\end{align}
The desired inequality follows by taking the logarithm of~\eqref{v infty R} and~\eqref{v mid infty R} and applying the bound~\eqref{|Pnl|}. 
\end{proof}

\section{Proof of Main theorem} \label{section: proof of main theorem}
This section is devoted to the proof of Theorem~\ref{main}. 
We use a linear independence criterion from \cite[Proposition $5.6$]{DHK3} \footnote{We easily see that the criterion \cite[Proposition~$5.6$]{DHK3} is also verified replacing $\lim_n \tfrac{1}{n}\sum_{v}F_v(n)<\infty$ by $\limsup_n \tfrac{1}{n}\sum_{v}F_v(n)<\infty$.}, which is based on the method of C.~F.~Siegel ({\it see} \cite{Siegel}).

We keep notations in Section~\ref{section: Pade multiple polylog} and \ref{section: estimates}.
First let us recall the necessary notation. Let $m,r$ be positive integers and $K$ be an algebraic number field. 
Fix a place $v_0 \in \mathfrak{M}_K$, $\boldsymbol{\alpha}=(\alpha_1,\ldots,\alpha_m)\in (K^{\times})^m$ whose coordinates are pairwise distnct and an element $\beta \in K$ with $|\beta|_{v_0} >H_{v_0}(\boldsymbol{\alpha})$. 
The quantity $V_{v_0}(\boldsymbol{\alpha},\beta)$ is defined in \eqref{V} for $v_0$. 

\medskip

Define the following $(m+1)^r \times (m+1)^r$ matrix $M_n$ as:
\begin{equation} \label{Mn}
M_n =  \begin{pmatrix}
  P_{n,\ell}(\beta) \\
  Q_{n,\boldsymbol{s},\boldsymbol{a},\ell}(\beta)
  \end{pmatrix}_{\substack{0\le \ell \le M \\ (\boldsymbol{s},\boldsymbol{a})\in \mathcal{S}}}
\in \mathrm{Mat}_{(m+1)^r}(K).
\end{equation}
By Corollary~\ref{multiple polylog det non zero}, the matrix $M_n$ is invertible.

\begin{proof}[\textbf{Proof of Theorem~\ref{main}}]
For $v \in \mathfrak{M}_K$, we define functions $F_v \colon \mathbb{N} \longrightarrow \mathbb{R}_{\ge 0}$ by
\begin{align*}
F_v(n) &= n \biggl(M h_v(\beta) + \dfrac{M}{m}\sum_{i=1}^mh_v(\alpha_i)+Mh_v(\boldsymbol{\alpha}) + \varepsilon_v \frac{d_v}{d} \left( M \log 2 + \dfrac{r(r+1)}{2}\log(m+1)+r \right) + o(1)\biggr)  \\
&\quad   + (\varepsilon_v-1) \log |d^r_{Mn+M}|_v,
\end{align*}
where $o(1)=0$ if $v\nmid \infty$. 
Proposition~\ref{estimate approximants} allows us to obtain the bound
\[
\log \max_{\substack{(\boldsymbol{s},\boldsymbol{a})\in \mathcal{S} \\ 0 \le \ell \le M}}  \{ |P_{n,\ell}(\beta)|_v, |Q_{n,\boldsymbol{s},\boldsymbol{a},\ell}(\beta)|_v \} \le F_v(n).
\]
Define a quantity
\[
\mathbb{A}_{v_0}(\beta)= n \biggl( -h_{v_0}(\beta) + \dfrac{M}{m}\sum_{i=1}^m h_{v_0}(\alpha_i) + Mh_{v_0}(\boldsymbol{\alpha}) + \varepsilon_{v_0} \frac{d_{v_0}}{d} \Bigl( M \log 2 + \frac{r(r+1)}{2} \log(m+1) + r \Bigr) + o(1) \biggr).
\]
Then Proposition~\ref{Remainder} yields
\[
\log \max_{\substack{(\boldsymbol{s},\boldsymbol{a})\in \mathcal{S} \\ 0 \le \ell \le M}} \{|\mathfrak{R}_{n,\boldsymbol{s},\boldsymbol{a},\ell}(\beta)|_{v_0}\} \le -\mathbb{A}_{v_0}(\beta)n + o(n).
\]
The prime number theorem (see \cite{R-S1}) implies
\[
\limsup_{n \to \infty} \frac{1}{n} \log d^r_{Mn+M} = rM,
\]
from which we obtain
\[
\mathbb{A}_{v_0}(\beta) - \limsup_{n \to \infty} \frac{1}{n} \sum_{v \neq v_0} F_v(n) \le V_{v_0}(\beta).
\]
Finally, we apply the linear independence criterion \cite[Proposition 5.6]{DHK3} for the values
\[
\theta_{\boldsymbol{s},\boldsymbol{a}} = f_{\boldsymbol{s},\boldsymbol{a}}(\beta) \quad \text{for} \ \ (\boldsymbol{s},\boldsymbol{a})\in \mathcal{S},
\]
the sequence of matrices $(M_n)_{n \ge 0}$ defined in \eqref{Mn} and $F_v$. Combined with the estimates above, we obtain Theorem~\ref{main}.
\end{proof}
\begin{proof}[\textbf{Proof of Corollary~\ref{main cor}}]
Let $s \in \mathbb{N}$. A straightforward computation yields the derivative
\[
\partial_z \cdot \mathrm{Li}_{s}(\alpha_i/z) =
\begin{cases}
-\dfrac{1}{z} \mathrm{Li}_{s-1}(\alpha_i/z) & \text{if } s > 1, \\[10pt]
-\dfrac{\alpha_i}{z(z-\alpha_i)} & \text{if } s = 1.
\end{cases}
\]
This equality together with Leibniz formula implies that the set of monomials
\[
\mathcal{T} := \{\L_{s_1}(\alpha_{i_1}/z)\ldots\L_{s_k}(\alpha_{i_k}/z) \mid   \boldsymbol{s}=(s_1,\ldots,s_k)\in \cup_{k=1}^r\N^k \ \text{with} \ |\boldsymbol{s}|\le r  \ \text{and} \  1\le i_j \le m\}
\]
is contained in the vector space $V_1(L)$. It is a well-known result that the set of functions 
\[
\{ \mathrm{Li}_s(\alpha_i/z) \mid 1 \le s \le r, \, 1 \le i \le m \}
\]
is algebraically independent over $K(z)$. 
In particular, the elements of $\mathcal{T}$ are linearly independent over $K$. 
Consequently, there exists a $K$-basis $\mathcal{B}$ of $V_1(L)$ that contains $\mathcal{T}$. 
Theorem~\ref{main} asserts that, under the assumption $V_{v_0}(\bm{\alpha}, \beta) > 0$, the set $\{1\} \cup \{ f(\beta) \mid f \in \mathcal{B} \}$ is linearly independent over $K$. It follows immediately that the subset 
\[
 \{1\} \cup \{\L_{s_1}(\alpha_{i_1}/\beta)\ldots\L_{s_k}(\alpha_{i_k}/\beta) \mid   \boldsymbol{s}=(s_1,\ldots,s_k)\in \cup_{k=1}^r\N^k \ \text{with} \ |\boldsymbol{s}|\le r  \ \text{and} \  1\le i_j \le m\}
\]
is linearly independent over $K$, as claimed.
\end{proof}

\section{Appendix: Pad\'{e}-type approximants for powers of logarithms} \label{section: Pade power of log}
Keeping the notation from Section \ref{rodrigues formula}, we fix a positive integer $m$ and specialize the results of Sections \ref{rodrigues formula} and \ref{section: linear independence pade}. 
Specifically, in a formal way, we construct Pad\'{e}-type approximants for the vector of logarithmic powers:
\[
\bigl(\log(1-1/z),\ldots,\log^{m}(1-1/z)\bigr)\in (1/z\cdot \Q[[1/z]])^{m}.
\]
To this end, we define the differential operator 
\[
L_m := \bigl(z(z-1)\partial_z\bigr)^{m} \in \Q[z,\partial_z].
\]
One easily verifies that $\ord_{(1,-1)}(L_m)=m$ and $L_m$ satisfies property $(P)$. Furthermore, we have
\[
V_1(L_m)={\rm{Span}}_{\Q}\{\log^{s}(1-1/z)\mid 1\le s \le m\}.
\]
\begin{theorem} \label{R_n of power of log}
Let $n$ be a positive integer.  
Then the differential operator
\[
R_n
  := \frac{1}{(n!)^{m}}\bigl(\,z^{n}(z-1)^n\partial_{z}^{n}\bigr)^{m}
  \in \Q[z,\partial_z]
\]
belongs to the left ideal $I_n(L_m)$ $($see Definition~$\ref{Rodrigues ideal})$.
\end{theorem}
To prove Theorem~\ref{R_n of power of log}, we prepare several lemmas. In the following of this section, for any positive integer $n$, we put
\[
E_n=z^n(z-1)^n\partial^n_z.
\]
\begin{lemma}\label{basic identities}
Let $n$ be a positive integer. Then, in the ring $K[z,\partial_z]$, the following identities hold.
\begin{enumerate}
      \item[{\rm{\rm(i)}}] $E_n=(E_1-(n-1)(2z-1))\cdots (E_1-(2z-1)) E_1$.
      \item[{\rm(ii)}] $E_{n+1}z=z(E_1-(n-1)z-1)E_n$.
\end{enumerate}
\end{lemma}
\begin{proof}
${\rm(i)}$ We show the desire identity by induction on $n$. In case $n=1$, the identity is trivial. Assume the identity holds for $n\ge1$. 
Using the induction hypothesis, we get
\begin{align*}
E_{n+1}&=z(z-1)\left(z^n(z-1)^n\partial_z\right)\partial^n_z\\
                                             &=\left(E_1-n(2z-1)\right)E_n\\
                                             &=\left(E_1-n(2z-1)\right) (E_1-(n-1)(2z-1))\cdots (E_1-(2z-1))E_1.
\end{align*}
${\rm(iii)}$ Using the identity
\[
\partial^{n+1}z=z\partial^{n+1}_z+(n+1)\partial^n_z,
\]
the left hand side of the desire identity can be computed
\begin{align*}
E_{n+1}z&=z^{n+1}(z-1)^{n+1}(z\partial^{n+1}_z+(n+1)\partial^n_z)\\
&=z\left(z(z-1)\left(z^n(z-1)^n\partial_z\right)\partial^n_z+(n+1)(z-1)E_n\right)\\
&=z\left(E_1-(n-1)z-1\right)E_n.
\end{align*}
This completes the proof of Lemma~\ref{basic identities}.
\end{proof}

\begin{lemma}\label{basic relation}
Let $n,s$ be positive integers. Then we have
\[
E_n\cdot z^{n-1}\log^s(1-1/z)\in \Q[z]\oplus \bigoplus_{j=1}^{s-1} \Q[z]_{\le n-1}\log^j(1-1/z),
\]
where $\bigoplus_{j=1}^{s-1} \Q[z]_{\le n-1}\log^j(1-1/z)=\{0\}$ if $s=1$.
\end{lemma}
\begin{proof}
We show the assertion by induction on $n$. In case of $n=1$, for any positive integer $s$, a straight forward computation yields 
\[
E_1\cdot \log^{s}(1-1/z)=s\log^{s-1}(1-1/z)\in \Q[z]\oplus \Q\log^{s-1}(1-1/z).
\]
Therefore the statement holds. Assume the relation holds for $n\ge1$. In case of $n+1$, applying Lemma~\ref{basic identities} ${\rm(ii)}$, 
\begin{align}
E_{n+1}\cdot z^{n}\log^s(1-1/z)&=E_{n+1}z\cdot z^{n-1}\log^s(1-1/z) \nonumber\\
                                                                               &=z(E_1-(n-1)z-1)E_n\cdot z^{n-1}\log^s(1-1/z). \label{compute}
\end{align}
By the induction hypothesis, there exist polynomials $P(z),P_1(z),\ldots,P_{s-1}(z)\in \Q[z]$ with $\deg\,P_j\le n-1$ such that
\begin{align} \label{expand En}
E_n\cdot z^{n-1}\log^s(1-1/z)=P(z)+\sum_{j=1}^{s-1}P_j(z)\log^{j}(1-1/z).
\end{align}
A straight forward computation yields 
\[
\deg\left(z(E_1-(n-1)z-1)\cdot P_j(z)\right)\le n,
\]
and using equations \eqref{compute} and \eqref{expand En}, we conclude
\[
E_{n+1}\cdot z^{n}\log^s(1-1/z)\in \Q[z]\oplus \bigoplus_{j=1}^{s-1} \Q[z]_{\le n}\log^j(1-1/z),
\]
as desired.
\end{proof}
\begin{proof}[\textbf{Proof of Theorem~\ref{R_n of power of log}}]
Since $\log(1-1/z)$ is transcendental over $\Q(z)$, we have
\[
V_n(L_m)
 = \Span_{\Q}\!\left\{ \pi(z^k \log^{s}(1-1/z))
   \,\middle|\, 0\le k\le n-1,\ 1\le s\le m \right\}.
\]
Hence it suffices to show that
\begin{equation}\label{aim}
\bigl(z^n(z-1)^n\partial_z^n\bigr)^m
 \cdot z^k\log^s(1-1/z)\in \Q[z]
\end{equation}
for all integers $0\le k\le n-1$ and $1\le s\le m$.
Fix an integer $k$ with $0\le k\le n-1$.
By Lemma~\ref{basic identities}\,{\rm(i)}, we have the factorization
\[
z^n(z-1)^n\partial_z^n
 = \prod_{j=k+1}^{n-1}
   \bigl(z(z-1)\partial_z - j(2z-1)\bigr)
   \cdot z^{k+1}(z-1)^{k+1}\partial_z^{k+1} .
\]
For a positive integer $s$, Lemma~\ref{basic relation} for $k+1$, together with the above identity, implies
\[
z^n(z-1)^n\partial_z^n \cdot z^k\log^s(1-1/z)
 \in \Q[z]
 \oplus
 \bigoplus_{j=1}^{s-1} \Q[z]_{\le n-1}\log^j(1-1/z).
\]
Iterating this argument, we obtain
\[
\bigl(z^n(z-1)^n\partial_z^n\bigr)^m
 \cdot z^k\log^s(1-1/z)\in \Q[z]
\qquad (1\le s\le m),
\]
which proves \eqref{aim} and completes the proof.
\end{proof}
In what follows, for the $K$-isomorphism $\Phi$ defined in \eqref{Phi}, we set
\[
\varphi_s:=\Phi\bigl(\log^{s}(1-1/z)\bigr) \qquad (1\le s\le m).
\]
The next result is a direct consequence of Theorem~\ref{R_n of power of log} and Proposition~\ref{Pade}.
\begin{corollary}\label{pade log}
Let $m$ be a positive integer, and let $n,\ell$ be nonnegative integers. Define the polynomials
\begin{align*}
P_{n,\ell}(z) = {\rm Eval}_{t=z}\circ R^{*}_n\cdot t^{\ell},\ \ Q_{n,s,\ell}(z)= \varphi_{s}\!\left(\frac{P_{n,\ell}(z)-P_{n,\ell}(t)}{z-t}\right) \qquad (1\le s\le m).
\end{align*}
Then the vector
$
\bigl(P_{n,\ell},\,Q_{n,1,\ell},\,\ldots,\,Q_{n,m,\ell}\bigr)
$
forms a weight $n$ Pad\'{e}-type approximant of 
\[
\bigl(\log(1-1/z),\ldots,\log^{m}(1-1/z)\bigr).
\]
\end{corollary}
\begin{proof}
By definition of $R_n$, we have 
\begin{align} \label{deg R^*_nP}
\deg \, R^{*}_n\cdot t^\ell=mn+\ell
\end{align}
and, in particular the polynomial $P_{n,\ell}(z)$ is nonzero. Proposition~\ref{Pade} implies that the statement holds.
\end{proof}
We keep the notation in Corollary~\ref{pade log}.
Define the determinant of the $(m+1)\times (m+1)$ matrix
\[
\Delta_{n}(z)
  := \det
  \begin{pmatrix}
  P_{n,0}(z) & P_{n,1}(z) & \cdots & P_{n,m}(z)\\
  Q_{n,1,0}(z) & Q_{n,1,1}(z) & \cdots & Q_{n,1,m}(z)\\
  \vdots & \vdots & \ddots & \vdots\\
  Q_{n,m,0}(z) & Q_{n,m,1}(z) & \cdots & Q_{n,m,m}(z)
  \end{pmatrix}.
\]
\begin{corollary}\label{power of log det non zero}
We have $\Delta_{n}(z)\in \Q^{\times}$.
\end{corollary}
\begin{proof}
By Proposition~\ref{equivalence nonzero det}, it suffices to show that 
\(
R_{n}
\)
satisfies properties ${\rm(i)}$ and ${\rm(ii)}$ in Proposition~\ref{equivalence nonzero det} and the Laurent series $1$ and $(\log^s(1-1/z))_{1\le s \le m}$ are linearly independent over $\Q(z)$.
Equation \eqref{sum of ord} together with the definition of $\ord_{(1,-1)}$ yields
\[
\ord_{(1,-1)}(R_{n}) = 2mn - mn = \ord_{(1,-1)}(L_{m})n,
\]
so $R_n$ satisfies property ${\rm(i)}$.  
Next, combining equation~\eqref{deg R^*_nP} with Lemma~\ref{equiv of (P)} implies that $R_n$ satisfies property $(P)$.  
Finally, since $\log(1-1/z)$ is transcendental over $\Q(z)$, $1$ and $(\log^s(1-1/z))_{1\le s \le m}$ are linearly independent over $\Q(z)$.
The proof of Corollary~\ref{power of log det non zero} is now complete. 
\end{proof}

\noindent
{\bf Acknowledgements.}

\smallskip

This work was supported in part by the Research Institute for Mathematical Sciences, an International Joint Usage/Research Center located in Kyoto University, and by the Institute for Mathematical Informatics, Meiji Gakuin University.
The author was supported by JSPS KAKENHI Grant Number 24K16905.
\bibliography{}

\medskip\vglue5pt
\vskip 0pt plus 1fill
\hbox{\vbox{\hbox{Makoto \textsc{Kawashima}}
\hbox{{\tt kawasima@mi.meijigakuin.ac.jp}}
\hbox{Institute for Mathematical Informatics}
\hbox{Meiji Gakuin University}
\hbox{Totsuka, Yokohama, Kanagawa}
\hbox{224-8539, Japan}
}}
\end{document}